\documentclass[english]{article} 
\usepackage[utf8]{inputenc}
\usepackage[T1]{fontenc}
\usepackage{lmodern}
\usepackage[a4paper]{geometry}
\usepackage{enumitem}
\usepackage{verbatim}
\usepackage{graphicx}
\usepackage{amsmath}
\usepackage{amsthm}
\usepackage{amsfonts}
\usepackage{amssymb}
\usepackage{amsmath}
\usepackage{mathrsfs}
\usepackage{bbm}
\usepackage{framed}
\usepackage{xcolor}
\usepackage[english]{babel}
\usepackage[citecolor=blue,colorlinks=true]{hyperref}
\hypersetup{pdfnewwindow}
\theoremstyle{plain}
\newtheorem{theorem}{Theorem}[section]
\newtheorem{proposition}[theorem]{Proposition}
\newtheorem{lemma}[theorem]{Lemma}

\theoremstyle{definition}
\newtheorem{definition}{Definition}[section]

\theoremstyle{remark}
\newtheorem{remark}{Remark}[section]

\numberwithin{equation}{section}
\newcommand{\N}{\mathbb{N}}
\newcommand{\Z}{\mathbb{Z}}

\newcommand{\R}{\mathbb{R}}

\def\ocirc#1{\ifmmode\setbox0=\hbox{$#1$}\dimen0=\ht0
    \advance\dimen0 by1pt\rlap{\hbox to\wd0{\hss\raise\dimen0
    \hbox{\hskip.2em$\scriptscriptstyle\circ$}\hss}}#1\else
    {\accent"17 #1}\fi} 

\newcommand{\eps}{\varepsilon}
\newcommand{\F}{\mathcal{F}}

\newcommand{\PP}{\mathbb{P}}

\newcommand{\dual}[2]{\langle #1, #2\rangle}
\newcommand{\E}{\mathbb{E}}
\newcommand{\T}{\mathbb{T}}

\newcommand{\LL}{\mathscr{L}}

\DeclareMathOperator{\esssup}{ess\,sup}

\newcommand{\nx}{\nabla_{\!\!x}\,} 
\newcommand{\cd}{\!\cdot\!} 
 
\newcommand{\axi}{a^\xi} 

\begin{document}

\title{Global solutions to quadratic systems of stochastic reaction-diffusion equations in space-dimension two.}
\author{M. Leocata\thanks{Scuola Normale Superiore, Pisa, Italy, marta.leocata@sns.it} and J. Vovelle\thanks{UMPA, CNRS, ENS de Lyon, julien.vovelle@ens-lyon.fr}}
\maketitle

\begin{abstract} We prove the existence of global-in-time regular solutions to a system of stochastic quadratic reaction-diffusion equations. Global-in-time existence is based on a $L^\infty$-estimate obtained by an approach \`a la De Giorgi, as in \cite{GoudonVasseur10}. The adaptation of this technique to the stochastic case requires in its final step an $L^2\ln(L^2)$-bound, furnished by an estimate by duality on the entropy inequality, as in \cite{DesvillettesFellnerPierreVovelle07}. In our stochastic context, and similarly to \cite{DebusscheRoselloVovelle2021}, we need to solve a backward SPDE to exploit the duality technique.
\end{abstract}
%

\normalsize

\tableofcontents
\bigskip

\section{Introduction}\label{sec:intro}
\subsection{Quadratic stochastic reaction-diffusion system}\label{subsec:RDsystem}

We consider the binary reversible chemical reaction
\begin{equation}\label{chemical}
A_1+A_3\rightleftharpoons A_2+A_4.
\end{equation}
The species $A_i$, $i=1,\ldots,4$ move by diffusion in the periodic domain $\T^d=\R^d/\Z^d$. Some statements below (entropy estimate, $L^2$-estimate by duality) are given in any space dimension $d$, but our main results ($L^\infty$-bounds, global-in-time solutions) is restricted to the space dimensions $d=1$ or $d=2$. The evolution of the concentrations $a_i(t)$, $i=1,\ldots,4$, of the species $A_i$, $i=1,\ldots,4$ is described by the following system of equations
\begin{equation}\label{SS}
\begin{split}
 d a_i-\nx\cd \left(\kappa_i\nx a_i\right)dt&=f_i(a)dt+\sigma_i^\alpha(a)d B_\alpha(t),\; {\rm in}\  \;Q_T:= \T^d\times(0,T),\\
a_i(0)&=a_{i0}\geq 0,\; {\rm in}\   \T^d,
\end{split}
\end{equation}
for $i=1,\ldots,4$, where $f)i(a)$ is the reaction term, given by 
\begin{equation}\label{fi-quadratic}
f_i(a)=(-1)^i(a_1 a_3-a_2 a_4).
\end{equation}
The family $\left\{ B_\alpha;1\leq\alpha\leq d_W\right\}$ is a finite family of independent one-dimensional Wiener processes of the filtered probability space $(\Omega,\F,\PP,(\F_t))$. Summation over the repeated index $\alpha$ is used in \eqref{SS}. Our modelling approach, explained in Section~\ref{subsec:modelling} below, shall lead us to consider coefficients $\sigma_i^\alpha(a)$ of the type
\begin{equation}\label{true-sigma-alpha}
\sigma^1_i(a)=(-1)^i\sqrt{a_1 a_3},\quad\sigma^2_i(a)=(-1)^i\sqrt{a_2 a_4},\quad\sigma_i^\alpha=0,\;\forall\alpha\geq 3.
\end{equation}
However, we need the cancellation condition 
\begin{equation}\label{cancel-sigma-alpha}
a_i=0\Rightarrow\sigma_i^\alpha(a)=0,\;\forall i,
\end{equation}
to ensure that the solutions to \eqref{SS} stay non-negative, so we assume that \eqref{true-sigma-alpha} is satisfied asymptotically only: for $\alpha\in\N\setminus\{0\}$, $\sigma_i^\alpha\colon\R^d\to\R$ is a smooth function satisfying \eqref{cancel-sigma-alpha} and the growth condition
\begin{equation}\label{growth-sigma-alpha2}
\sum_\alpha|\sigma_i^\alpha(a)|^2\leq \nu(a_1 a_3+a_2a_4),
\end{equation}
where $\nu$ is a positive constant. We also assume that 
\begin{equation}\label{positivekappa}
\underline{\kappa}:=\min_{1\leq i\leq 4}\kappa_i>0.
\end{equation}
We assume the diffusion coefficients $\kappa_i$ to be constant, but variable (deterministic) coefficients depending on the variable $(t,x)$ may also be considered, as long as the following bounds are satisfied:
\begin{equation}\label{barkappaT}
\sup_{1\leq i\leq 4}\left(\|\nx\kappa_i\|_{L^\infty(Q_T)}+\|\kappa_i\|_{L^\infty(Q_T)}+\|\kappa_i^{-1}\|_{L^\infty(Q_T)}\right)<+\infty.
\end{equation}
We note also that we may as well consider different, independent noises for each equation in \eqref{SS}, as in \eqref{SSapp}. Let us remark further that the restriction $d_W<+\infty$, \textit{i.e.} the fact that we work with a finite-dimensional Wiener process, is relevant only in the justification of Theorem~\ref{th:RegSolSyst} in appendix~\ref{sec:app1}. \medskip

\begin{remark}[Boundary conditions]\label{rk:DirBC} We consider periodic conditions to avoid the compatibility problems that arise in parabolic SPDEs when Dirichlet conditions (even homogeneous Dirichlet conditions) are considered (all the more since, for a system of equations, we would have to consider a system of compatibility conditions, for which existence of solution is not guaranteed without further analysis - at least if one would not consider compactly supported initial data), \textit{cf.} \cite{Flandoli1990,DebusscheDeMoorHofmanova15,Gerencser2019}. It would also be very natural to consider homogeneous Neumann boundary conditions, but this is a framework for which the adequate references would be missing both for the direct parabolic stochastic problems and for the backward SPDEs used in Section~\ref{sec:L2duality}.
\end{remark}

Our main result and some additional comments and bibliographical references are given in Section~\ref{subsec:mainresult}.

\subsection{Notion of solution}\label{subsec:solution}

Let us first introduce a notion of weak solution.

\begin{definition}[Weak solution up to a stopping time] Let $a_0\in L^2(\T^d;\R^q))$ and let $\tau$ be a stopping time, $\tau>0$ a.s. A predictable $L^2(\T^d;\R^4)$-valued process $(a(t))_{t<\tau}$ is said to be a weak solution to the system \eqref{SS} up to time $\tau$ if
\begin{enumerate}
\item $\PP$-a.s., for all $t\in[0,\tau)$, $a_{i}(t)\geq 0$, for all $i=1,\dotsc,q$, 
\item\label{weaksol-L2} for all $T>0$, the stopped processes $a_i^T\colon t\mapsto a_i(t\wedge T)$ satisfy, for all $i=1,\dotsc,q$,  
\begin{equation}
a_i^T\in L^2(\Omega\times[0,T];H^1(\T^d))\cap L^2(\Omega; C([0,T];L^2(\T^d))),
\end{equation}
and
\begin{equation}
f_i(a^T)\in L^2(\Omega\times Q_T),\quad \sigma_{\alpha,i}(a^T)\in L^2(\Omega\times Q_T),
\end{equation}
\item\label{weaksol-weakeq} for all $\varphi\in H^1(\T^d)$, for all $i\in\{1,\dotsc,q\}$, $a_i$ satisfies $\PP$-almost surely, for all $t\in [0,\tau)$,   
\begin{multline}\label{weakFormulation}
\dual{a_i(t)}{\varphi}-\dual{a_{i0}}{\varphi}+\kappa_i\int_0^t \dual{\nabla_x a_i(s)}{\nabla_x\varphi}ds\\
=\int_0^t \dual{f_i(a(s))}{\varphi} ds+\int_0^t \sum_{\alpha=1}^{d_W}\dual{\sigma_{\alpha,i}(a(s))}{\varphi}dB_\alpha(s).
\end{multline}
\end{enumerate} 
\label{def:weak-sol}\end{definition}

Notations: we denote by $|a|$ the euclidean norm of a vector $a\in\R^4$:
\begin{equation}\label{l2R4}
|a|=\left(\sum_{i=1}^4 a_i^2\right)^\frac12.
\end{equation}

\begin{definition}[Regular solution up to a stopping time] 
Assume $a_{i0}\in C^\infty(\T^d)$, $i\in\{1,\ldots,4\}$. Let $\tau$ be a stopping time, $\tau>0$ a.s. A weak solution $(a(t))_{t<\tau}$ defined up to time $\tau$ is said to be a regular solution of \eqref{SS} up to time $\tau$ if, denoting by $\tau_n$ the stopping time\footnote{with the convention $\tau_n=\tau$ if $\esssup_{x\in \T^d} |a(x,t)|\leq n$ for all $t\in[0,\tau)$}
\begin{equation}\label{taun0}
\tau_n=\inf\left\{t\in[0,\tau);\esssup_{x\in \T^d} |a(x,t)|>n\right\},
\end{equation}
the following properties are satisfied: for all $p\in[2,\infty)$, $q\in(2,\infty)$, $m\in\N$ with $m\geq 2$ and $mp>d+2$, and all $T>0$, the stopped processes $a^{\tau_n\wedge T}\colon t\mapsto a(t\wedge\tau_n\wedge T)$ satisfy
\begin{equation}
a^{\tau_n\wedge T}\in L^q(\Omega;C([0,T];W^{m,p}_0(\T^d)))\cap L^{mq}(\Omega;C([0,T];W^{1,mp}_0(\T^d))).
\end{equation}
A regular solution up the stopping time $\tau\equiv+\infty$ is said to be a global-in-time regular solution.
\end{definition}

\paragraph{Construction of a regular solution up to a stopping time.} Introduce the truncated non-linearities
\begin{equation}\label{truncfsigma}
f^n_i(a)=|\chi_n(|a|)|^2 f_i(a),\quad \sigma_i^{\alpha,n}(a)=\chi_n(|a|) \sigma_i^\alpha(a),
\end{equation}
where 
\begin{equation}\label{truncationTn}
\chi_n(r)=\chi(n^{-1}r),
\end{equation}
the function $\chi\colon\R_+\to\R+$ being smooth, non-increasing, and such that $\chi(r)=1$ if $r\leq 1$, $\chi(r)=0$ if $r\geq 2$. We let $a^{[n]}$ denote the global-in-time regular solution to \eqref{SS} where $f_i$ is replaced by $f_i^n$ and $\sigma_i^\alpha$ is replaced by $\sigma^{\alpha,n}_i$. The existence of  $a^{[n]}$ is ensured by Theorem~\ref{th:RegSolSyst} in Appendix~\ref{sec:app1}. We then set 
\begin{equation}\label{taun}
\hat{\tau}_n=\inf\left\{t\geq 0;\esssup_{x\in \T^d} \|a^{[n]}(x,t)\|>n\right\}.
\end{equation}
Since $f^{n+1}_i(a)=f^n_i(a)$, $\sigma^{\alpha,n+1}_i(a)=\sigma^{\alpha,n}_i(a)$ if $|a|\leq n$, the functions
\begin{equation}\label{taunn}
t\mapsto a^{[n]}(t\wedge\hat{\tau}_n\wedge\hat{\tau}_{n+1})\mbox{ and } t\mapsto a^{[n+1]}(t\wedge\hat{\tau}_n\wedge\hat{\tau}_{n+1})
\end{equation}
satisfy the same equation. By uniqueness (Theorem~\ref{th:RegSolSyst}) they coincide,
with the consequence that $\hat{\tau}_n\leq\hat{\tau}_{n+1}$, and $a^{[n]}=a^{[n+1]}$ a.s. on $[0,\hat{\tau}^n]$. We can the consider the limit
\begin{equation}\label{stoppingtau}
\tau=\lim_{n\to+\infty}\hat{\tau}_n.
\end{equation}
Then $\tau$ is a stopping time such that $\tau>0$ a.s. and the process $(a(t))$ defined by
\begin{equation}\label{defatau}
a(t)=\begin{cases}
a^{[n]}(t) & \mbox{if } t\leq\hat{\tau}_n\\
0 &\mbox{if } t>\tau
\end{cases}
\end{equation}
is a regular solution to \eqref{SS} up to the stopping time $\tau$.

\subsection{Main result}\label{subsec:mainresult}

\begin{theorem}[Global solutions]\label{th:globalregular} 
Suppose that the coefficients $\sigma_i^\alpha$ satisfy \eqref{cancel-sigma-alpha} and \eqref{growth-sigma-alpha2} and that the space dimension is $d=1$ or $d=2$. Suppose that the positivity condition \eqref{positivekappa} is satisfied by the diffusion coefficients $\kappa_i$. Then there is a constant $C_B$ depending on $d$, on $\min_{1\leq i\leq 4}\kappa_i$ $\max_{1\leq i\leq 4}\kappa_i$ only such that, if the noise is sufficiently small, in the sense that the constant $\nu$ in \eqref{growth-sigma-alpha2} satisfies
\begin{equation}\label{smallnoise}
C_B\nu\leq 1,
\end{equation}
then, for all $a_{i0}\in C^\infty({\T^d})$ with $a_{i0}\geq 0$, $i=1,\ldots,4$, the problem \eqref{SS} admits a unique global-in-time regular solution.
\end{theorem}

In the deterministic case, the existence of global-in-time regular solutions to quadratic systems of reaction-diffusion equation has been established by various authors. Souplet, \cite{Souplet2018}, and Fellner, Morgan, Tang, \cite{FellnerMorganTang2020}, use the maximum principle for certain auxiliary equations, together with regularity estimates for scalar parabolic equations. This approach seems difficult to adapt to stochastic equations (when multiplicative noise is considered at least, the ``Da Prato - Debussche trick'' allowing to treat additive noise). De Giorgi's iteration scheme, based on truncates of the entropy, is shown to be efffective (for providing $L^\infty$-bounds almost-everywhere) in dimension $d\leq 2$ by Goudon and Vasseur, \cite{GoudonVasseur10}. The restriction on the dimension is related to the use of Sobolev's inequality in De Giorgi's iteration method. By adding several non-trivial ingredients (blow-up argument, estimates for the mass equation in particular), Caputo Goudon and Vasseur extend in \cite{CaputoGoudonVasseur2019} the result from dimension $d\leq 2$ to any space dimension. Here, and with the help of the $L^2\ln(L^2)$-estimate given in Section~\ref{sec:L2duality},we show that  the approach of \cite{GoudonVasseur10} is effective, under the same constraint $d\leq 2$.  \medskip

There are few works on systems of reaction-diffusion equations with stochastic forcing in the literature. Let us mention 
\cite{Kunze2015,DhariwalJuengelZamponi2019,BendahmaneKarlsen2022,DuYeZhang2023} and \cite{HausenblasRandrianasoloThalhammer20} however. The papers \cite{Kunze2015,DhariwalJuengelZamponi2019,BendahmaneKarlsen2022,DuYeZhang2023} are related to our work, but relatively different in nature: the emphasis is more on the existence of weak solutions (the problem of uniqueness is addressed in \cite{Kunze2015} also). In \cite{HausenblasRandrianasoloThalhammer20} a stochastic perturbation of the $2\times 2$ Gray-Scott model is considered on the Torus in dimension $d\leq 3$, and solutions in a Sobolev space that injects in $L^\infty$ are obtained.
However, one of the two equations in the Gray-Scott model has a ``good'' non-linear reaction term, insofar as it is non-positive (when the solutions stay in the class of solutions with non-negative components). This allows to start with an a priori estimate for the good component, a situation which has no counterpart in the case of the quadratic system \eqref{SS}.\medskip

Our method of proof uses a $L^2\ln(L^2)$ estimate as in \cite{DesvillettesFellnerPierreVovelle07}. An argument of duality is used to establish this $L^2\ln(L^2)$ estimate. In our stochastic context, we have to consider a backward SPDE as dual equation. Such a duality approach via BSPDE was already exploited in \cite{DebusscheRoselloVovelle2021}.\medskip

The $L^\infty$ bound on solutions is obtained by the method of De Giorgi, via truncation in the entropy inequality, as performed in \cite{GoudonVasseur10}. In the companion paper \cite{LeocataVovelle2023}, this method is analysed in the framework of linear parabolic equations. An alternative proof by duality of supremum estimates in this context is also given, and we still refer to this paper for references about supremum estimates for parabolic stochastic equations.\medskip

We conclude this paragraph with some details on the structure of the paper. The following section~\ref{subsec:modelling} accounts for the stochastic terms in the modelling of chemical reactions. The $L^1\ln(L^1)$ entropy estimate, Proposition~\ref{prop:entropy-estimate-quad-p}, and the $L^2\ln(L^2)$ estimate, Theorem~\ref{th:l2estim}, are derived in Section~\ref{sec:entropyestimate} and Section~\ref{sec:L2duality} respectively. The $L^\infty$ estimate on solutions, Theorem~\ref{th:Linfty-estimate}, is then obtained in Section~\ref{sec:LinftyRD}.

\subsection{Diffusion-approximation and modelling of the chemical reaction}\label{subsec:modelling}

In this section, we explain how an asymptotic expansion at the diffusive scale on the generator associated to the Markov description of the reaction \eqref{chemical} leads to \eqref{SS}. We will neglect the spatial displacements of the reactants in this discussion. 
\medskip

In a stochastic modelling of \eqref{chemical}, each reaction happens at a random time, given as an exponential random variable. The parameters of these exponential random variables, also called \emph{transition rates} of the reaction, are respectively of the form
\[
r_{\to}=\lambda_{\to}\frac{1}{N}N_1N_3,\quad r_{\leftarrow}=\lambda_{\leftarrow}\frac{1}{N}N_2N_4,
\]
where $N$ is the total number of reactants and $N_i$ the number of molecules of type $i$, \cite[p.454-455]{EthierKurtz86}. Note that $N$ stays constant in the evolution. We take it as main parameter and consider the situation where it is large. Using the approach\footnote{approach which is standard, see \cite{NziPardouxYeo2021,DebusscheNguepedjaNankep2021} for instance} of \cite[p.455]{EthierKurtz86}, we consider the Markov process $\hat A_N=(N_i(t))_{1\leq i\leq 4}$ with state space $\N^4$, given by
\begin{equation}\label{hatA}
\hat A_N(t)=\hat A_N(0)+\sum_{\ell} \ell E_\ell\left(\int_0^t N\eta_\ell\left(\frac{\hat A_N(s)}{N}\right)ds\right),
\end{equation}
In \eqref{hatA}, the sum is over the two following indices (where $*$ indicates transposition)
\begin{equation}\label{lto}
\ell_\to=( - 1 , 1 , - 1 , 1)^*,\quad \ell_\leftarrow=-\ell_\to,
\end{equation}
the processes $E_\ell$ are two independent unit Poisson processes and, given $a\in (\R_+)^4$, the functions $\eta_\ell(a)$ are defined by
\[
\eta_{\ell_\to}(a)=a_1a_3,\quad \eta_{\ell_\leftarrow}(a)=a_2a_4.
\]
The generator $\LL_N$ of the rescaled process $A_N(t)=N^{-1}\hat A_N(t)$ acts on functions $\varphi\colon(\R_+)^4\to\R$ and is given by
\[
\LL_N\varphi(a)=\sum_\ell N\eta_\ell(a)(\varphi(a+\ell N^{-1})-\varphi(a)).
\]
We introduce $f(a):=\sum_\ell \ell\eta_\ell(a)$ (it coincides with \eqref{fi-quadratic}) to expand $\LL_N\varphi(a)$ as follows
\begin{align}
\LL_N\varphi(a)&=f(a)\cdot D_a\varphi(a)+\sum_\ell N\eta_\ell(a)(\varphi(a+\ell N^{-1})-\varphi(a)-\ell N^{-1}D_a\varphi(a))\nonumber\\
&=f(a)\cdot D_a\varphi(a)+\frac{1}{N}\sum_\ell \eta_\ell(a)\ell\ell^* : D^2_a\varphi(a)+\mathcal{O}\left(\frac{1}{N^2}\right),\label{ExpandLLn}
\end{align}
where 
$A : B:=\sum_{i,j}A_{ij}B_{ij}$. At the order $0$, \eqref{ExpandLLn} gives the generator $f\cdot D_a\varphi$ associated to the ODE $\dot{a}=f(a)$. At order $1$, we obtain the generator 
\[
\LL_N^1\varphi(a)=f(a)\cdot D_a\varphi(a)+\frac{1}{N}\sum_\ell \eta_\ell(a)\ell\ell^* : D^2_a\varphi(a),
\]
associated to the SDE
\begin{equation}\label{SDEa}
da(t)=f(a(t))dt+\sqrt{\frac{2}{N}}\sum_\ell \ell [\eta_\ell(a(t))]^{1/2} dB_\ell(t),
\end{equation}
where the $(B_\ell)_\ell$ are independent one-dimensional Wiener processes. When the space variable is neglected, \eqref{SDEa} corresponds to \eqref{SS} where the sum over $\alpha$ involves two indexes, and 
\begin{equation}\label{A1234}
f(a)=(a_1a_3-a_2a_4)\ell,\quad \sigma_{1}(a)=\eps\sqrt{a_1a_3}\ell,\quad\sigma_2(a)=\eps\sqrt{a_2a_4}\ell,
\end{equation}
where $\eps$ is a small constant and $\ell=\ell_{\to}$, as given in \eqref{lto}.

\begin{remark}[Space-time white noise]\label{rk:spacetimewhite} When spatial displacement is taken into account, a pro\-ba\-bi\-lis\-tic description of the reaction \eqref{chemical} in terms of particles can be done, with an explicit expression for the associated generator (see \cite[Proposition~3.1]{LimLuNolen2020} for instance). A procedure of diffusion-approximation then leads to a stochastic system of reaction-diffusion equations with space-time white noises affecting both spatial derivatives and reaction terms, see\footnote{Note however that \cite{DeMasiFerrariLebowitz1986} gives a result in the vein of a central limit theorem, or of ``first-order approximation'', to use the term used in \cite{ChevallierOst2020} for instance. In this approach, the stochastic  equation obtained at the limit is linear (and involves the linearization of the zero-th order deterministic equation). In the framework of diffusion-approximation, where we use an expansion of the generator, and not of the solution, the stochastic equation obtained at the end is non-linear.} \cite[Theorem~3]{DeMasiFerrariLebowitz1986} for a similar (scalar) equation, describing the fluctuations of a Glauber-Kawasaki dynamics with a fast rate of exchanges. In the present study, we will limit ourselves to the consideration of coloured-in-space, time-white noise, the case of space-time white noise being singularly more difficult of course.
\end{remark}

\section{Entropy estimate}\label{sec:entropyestimate}


We use the notation $s^*=1+s$, $s\in[0,+\infty)$. Let us set
\begin{equation}\label{EntropyBinary1}
\Phi(s)=s^*\ln(s^*)-s^*+1=(1+s)\ln(1+s)-s,\quad\bar{\Phi}(s)=\Phi(s)+s=(1+s)\ln(1+s).
\end{equation}
Although $\Phi$ is the natural entropy functional, we need to introduce $\bar{\Phi}$, simply to ensure a control of quantities $\sim c_0 s$ with $c_0\not=0$ around $s=0$.

\begin{proposition}[Entropy estimate]\label{prop:entropy-estimate-quad-p} Suppose that the coefficients $\sigma_i^\alpha$ satisfy \eqref{cancel-sigma-alpha} and \eqref{growth-sigma-alpha2}. Assume also $a_{i0}\in C^\infty({\T^d})$, $a_{i0}\geq 0$, $i=1,\ldots,4$. Let $a$ be a global-in-time regular solution to \eqref{SS} or a global-in-time solution to \eqref{SS} with the truncated non-linearities defined in \eqref{truncfsigma}-\eqref{truncationTn}. Then $a$ satisfies the following estimates: for all $p\in[1,+\infty)$, for all $T>0$,
\begin{equation}\label{entropyestimatep}
\E\left[\left(\sum_{i=1}^4\int_{\T^d}\bar{\Phi}(a_i(t))dx+\E\sum_{i=1}^4\int_0^t\int_{\T^d}\kappa_i \frac{|\nabla_x a_i|^2}{1+a_i}dxds\right)^p\right]
\leq C(p)\E\left[\left(\sum_{i=1}^4\int_{\T^d}\bar{\Phi}(a_{i0})dx\right)^p\right],
\end{equation}
for all $t\in[0,T]$, where the constant $C(p)$ depends on $p$, $d$, on $\nu$, $T$ and on the constant $\underline{\kappa}$ in \eqref{positivekappa}.
\end{proposition}

\begin{proof}[Proof of Proposition~\ref{prop:entropy-estimate-quad-p}] We will denote by $C_1,C_2,\ldots$ any constant depending on $d$, $\nu$ and $\underline{\kappa}$ only. We write $C_1(p),C_2(p),\ldots$ if there is also a dependence on $p$. We will establish \eqref{entropyestimatep} under the additional assumption that 
\begin{equation}\label{smallTp}
T\leq T_1(p),
\end{equation} 
for a positive time $T_1(p)$ depending on $p$, $d$ and $\nu$ (\textit{cf.} \eqref{growth-sigma-alpha2}). The general case will follow by iteration of this result. We consider first the case of a global-in-time regular solution $a$ of \eqref{SS}. The case of truncated non-linearities is explained at the end of the proof. \medskip

By the It\^o formula, we have for all $i$:
\begin{multline}\label{eqzi}
d\bar{\Phi}(a_i)-\nx\cd(\kappa_i\nx\bar{\Phi}(a_i))dt\\
= \left(\bar{\Phi}'(a_i)f_i(a)-\kappa_i\bar{\Phi}''_i(a_i)|\nabla_x a_i|^2+\frac12 \sum_\alpha|\sigma_i^\alpha(a_i)|^2\bar{\Phi}''_i(a_i)\right) dt
+\bar{\Phi}'_i(a_i)\sigma_i^\alpha(a)d B_\alpha(t).
\end{multline}
Then by integration in $x$, we obtain
\begin{multline}\label{eqziIntegrate}
\int_U \bar{\Phi}(a_i)(t)dx=\int_U \bar{\Phi}(a_i)(0)dx\\
+\int_0^t\int_U \left(\bar{\Phi}'(a_i)f_i(a)-\kappa_i\bar{\Phi}''_i(a_i)|\nabla_x a_i|^2+\frac12 \sum_\alpha|\sigma_i^\alpha(a_i)|^2\bar{\Phi}''_i(a_i)\right) dxds\\
+\int_0^t\int_U\bar{\Phi}'_i(a_i)\sigma_i^\alpha(a) dx d B_\alpha(s),
\end{multline}

\begin{lemma}[Control of the source terms]\label{lem:ControlSourceEntropy} There is a constant $C_1$ depending on $\nu$ only such that
\begin{equation}\label{eq:controlSource}
\sum_{i=1}^4 \left(\bar{\Phi}'(a_i)f_i(a)+\frac12 \sum_\alpha|\sigma_i^\alpha(a_i)|^2\bar{\Phi}''_i(a_i)\right)\leq C_1\sum_{i=1}^4 \bar{\Phi}(a_i),
\end{equation}
for all $a_1,\ldots,a_4\in\R_+$.
\end{lemma}

Summing on $i$ in \eqref{eqziIntegrate} and using the positivity condition \eqref{positivekappa} and the estimate \eqref{eq:controlSource}, we obtain
\begin{equation}\label{entropy1p}
\mathcal{E}(t)+\mathcal{D}(t)\leq C_1\int_0^t\mathcal{E}(s)ds+\mathcal{M}(t)+\mathcal{E}(0),
\end{equation}
where
\begin{equation}\label{anovEDp}
\mathcal{E}(t)=\sum_{i=1}^4\int_{\T^d}\bar{\Phi}(a_i(t))dx,\quad \mathcal{D}(t)=\E\sum_{i=1}^4\int_0^t\int_{\T^d} \frac{|\nabla a_i^\xi|^2}{a_i^{\xi*}}dxds,
\end{equation}
and $\mathcal{M}(t)$ denotes the martingale
\begin{equation}\label{noavMp}
\mathcal{M}(t)=\int_0^t\int_{\T^d} \sum_{\alpha,i}\sigma_i^\alpha(a_i)\left(1+\ln(a^{*}_i)\right) dx d B_\alpha(s).
\end{equation}
Our aim is to get the following estimate:
\begin{equation}\label{entropyestimatep-bis}
\E\left[\mathcal{U}(t)^p\right]\leq C(p) \E\left[\mathcal{E}(0)^p\right],\quad \mathcal{U}(t):=\sup_{0\leq s\leq t}\mathcal{E}(s)+\mathcal{D}(t),
\end{equation}
for $0\leq t\leq T_1(p).$ The bound 
\begin{equation}
C_1\int_0^t\mathcal{E}(s)ds\leq C_1 t\, \mathcal{U}(t)
\end{equation}
inserted in \eqref{entropy1p} gives us
\begin{equation}\label{entropyestimatep1}
\mathcal{U}(t)\leq 2\mathcal{E}(0)+2\mathcal{M}(t)^*,
\end{equation}
where $\mathcal{M}(t)^*=\sup_{s\in[0,t]}|\mathcal{M}(s)|$, for $0\leq t\leq T_1(p)$, provided $C_1 T_1(p)\leq \frac12$. Raising \eqref{entropyestimatep1} to the power $p$ yields
\begin{equation}\label{entropyestimatep2}
\mathcal{U}(t)^p\leq C_2(p)\left[\mathcal{E}(0)^p+\left(\mathcal{M}(t)^*\right)^p\right].
\end{equation}
We take expectation in \eqref{entropyestimatep2} and use the Burkholder-Davis-Gundy inequality to get
\begin{equation}\label{entropyestimatep3}
\E\left[\mathcal{U}(t)^p\right]\leq C_3(p)\left(\E\left[\mathcal{E}(0)^p\right]+\E\left[\dual{\mathcal{M}}{\mathcal{M}}_t^\frac{p}{2}	\right]\right).
\end{equation}
The quadratic variation of $\mathcal{M}(t)$ is
\begin{equation}\label{QvarMt}
\dual{\mathcal{M}}{\mathcal{M}}_t=\int_0^t\sum_\alpha\left|\int_{\T^d} \sum_{i}\sigma_i^\alpha(a)\left(1+\ln(a^*_i)\right) dx\right|^2 ds.
\end{equation}
Set $\mathbf{F}=a_1a_3+a_2a_4$. By the Cauchy-Schwarz inequality and \eqref{growth-sigma-alpha2}, we have
\begin{align}
\dual{\mathcal{M}}{\mathcal{M}}_t&\leq\int_0^t\left[\int_{\T^d} \sum_{\alpha,i}\mathbf{F}^{-1/2}|\sigma_i^\alpha(a)|^2\left(1+\ln(a^*_i)\right) dx 
\int_{\T^d} \sum_i\mathbf{F}^{1/2}\left(1+\ln(a^*_i)\right)dx\right]ds\nonumber\\
&\leq \nu\int_0^t\left|\int_{\T^d} \sum_i\mathbf{F}^{1/2}\left(1+\ln(a^*_i)\right)dx\right|^2 ds
\leq C_4 \int_0^t\left|\int_{\T^d} \sum_{i,j}a_i\left(1+\ln(a^*_j)\right)dx\right|^2 ds.\label{QvarMt2}
\end{align}
Observe that we have the bound $a\leq\bar{\Phi}(a)$ and
\begin{equation}\label{aiaj}
a_i^*\ln(a_j^*)\leq a_i^*\ln(a_i^*)+a_j^*\ln(a_j^*),
\end{equation}
which, exploited in \eqref{QvarMt2}, yields
\begin{equation}\label{QvarMt3}
\dual{\mathcal{M}}{\mathcal{M}}_t \leq C_5\int_0^t\left|\int_{\T^d} \sum_{i} \bar{\Phi}(a_i) dx\right|^2 ds\leq C_5 t\, \sup_{0\leq s\leq t}\mathcal{E}(s)^2\leq C_5 t\,\mathcal{U}(t)^2.
\end{equation}
We insert this last estimate \eqref{QvarMt3} in \eqref{entropyestimatep3} to obtain
\begin{equation}\label{entropyestimatep4}
\E\left[\mathcal{U}(t)^p\right]\leq C_6(p)\left(\E\left[\mathcal{E}(0)^p\right]+t^{\frac{p}{2}}\E\left[\mathcal{U}(t)^p\right] \right).
\end{equation}
The desired bound \eqref{entropyestimatep-bis} follows, under the condition $t\leq T_1(p)$.


\begin{proof}[Proof of Lemma~\ref{lem:ControlSourceEntropy}] Our aim is to control
\begin{equation}\label{introS}
S=-(a_1a_3-a_2a_4)\left(\ln(a_1^{*}a_3^{*})-\ln(a_2^{*} a_4^{*})\right)+\frac12 \sum_{\alpha,i}\frac{|\sigma_i^\alpha(a_i)|^2}{a_i^{*}}.
\end{equation}
By the growth condition \eqref{growth-sigma-alpha2}, we have 
\begin{equation}\label{estimateES}
S\leq 
-(a_1^*a_3^*-a_2^*a_4^*)\left(\ln(a_1^*a_3^*)-\ln(a_2^* a_4^*)\right)+\nu\sum_{i=1}^4\frac{a_1^*a_3^*+a_2^*a_4^*}{a_i^*}
+\mathtt{e}_1,
\end{equation}
where
\begin{equation}\label{remainderES}
\mathtt{e}_1= (a_1^*+a_3^*-a_2^*-a_4^*)\left(\ln(a_1^*a_3^*)-\ln(a_2^* a_4^*)\right).
\end{equation}
We use the inequality \eqref{aiaj} to obtain the bound
\begin{equation}\label{bounde1}
\mathtt{e}_1\leq C_2\sum_{i=1}^4\bar{\Phi}(a_i).
\end{equation}
We also have
\begin{equation}\label{boundCorrecIto0}
\sum_{i=1}^4\frac{a_1^*a_3^*+a_2^*a_4^*}{a_i^*} \leq C_3 \sum_{i=1}^4(\Phi(a_i)+a_i)
+ \left(\frac{a_1^*a_3^*}{a_2^*}+\frac{a_1^*a_3^*}{a_4^*}+\frac{a_2^*a_4^*}{a_1^*}+\frac{a_2^*a_4^*}{a_3^*}\right),
\end{equation}
so
\begin{equation}\label{estimateES2}
S \leq \Theta(a) 
+\mathtt{e}_2,
\end{equation}
where $\mathtt{e}_2$ satisfies the same bound \eqref{bounde1} as $\mathtt{e}_1$ and where
\begin{equation}\label{integrandSS}
\Theta(a)=-(a_1^*a_3^*-a_2^*a_4^*)\left(\ln(a_1^*a_3^*)-\ln(a_2^* a_4^*)\right)
+\nu\left[\frac{a_1^*a_3^*}{a_2^*}+\frac{a_1^*a_3^*}{a_4^*}+\frac{a_2^*a_4^*}{a_1^*}+\frac{a_2^*a_4^*}{a_3^*}\right].
\end{equation}
We will conclude the proof by showing that
\begin{equation}\label{Theta-by-E}
\Theta(a)\leq C_4\sum_{i=1}^4\bar{\Phi}(a_i).
\end{equation}
Let $K>1$ denote a constant that will be fixed later. Let us examine $\Theta(a)$ in the three regions
\begin{equation}
\Upsilon_+=\{a_2^*a_4^*>Ka_1^*a_3^*\},\quad \Upsilon_0=\{Ka_1^*a_3^*\geq a_2^* a_4^*\geq K^{-1}a_1^*a_3^*\},\quad \Upsilon_-=\{a_1^*a_3^*>K a_2^*a_4^*\}.
\end{equation}
For symmetry reasons, it is sufficient to examine the last two regions $\Upsilon_0$, $\Upsilon_-$. In $\Upsilon_0$, we simply use the sign of the entropy dissipation term,
\begin{equation}
-(a_1^*a_3^*-a_2^*a_4^*)(\ln(a_1^*a_3^*)-\ln(a_2^*a_4^*))\leq 0,
\end{equation}
and the bound
\begin{equation}
\frac{a_1^*a_3^*}{a_2^*}+\frac{a_1^*a_3^*}{a_4^*}+\frac{a_2^*a_4^*}{a_1^*}+\frac{a_2^*a_4^*}{a_3^*}
\leq K\left(a_4^*+a_2^*+a_3^*+a_1^*\right).
\end{equation}
In $\Upsilon_-$, we can estimate $\Theta(a)$ from above as follows: let us write $a_2^*a_4^*=\alpha a_1^*a_3^*$ with $\alpha<K^{-1}$. Then we obtain
\begin{equation}
(a_1^*a_3^*-a_2^*a_4^*)(\ln(a_1^*a_3^*)-\ln(a_2^*a_4^*))= a_1^*a_3^*\psi(\alpha),
\end{equation}
where $\psi(\alpha)=-(1-\alpha)\ln(\alpha)$ is positive decreasing on $(0,1]$ with $\psi(0+)=+\infty$. In particular, $\psi(\alpha)>\psi(K^{-1})$ and thus
\begin{equation}
-(a_1^*a_3^*-a_2^*a_4^*)(\ln(a_1^*a_3^*)-\ln(a_2^*a_4^*))\leq -a_1^*a_3^*\psi(K^{-1}).
\end{equation}
The remaining part of $\Theta(a)$ is bounded by
\begin{equation}
\frac{a_1^*a_3^*}{a_2^*}+\frac{a_1^*a_3^*}{a_4^*}+K^{-1}a_3^*+K^{-1}a_1^*
\leq 2a_1^*a_3^*+K^{-1}a_3^*+K^{-1}a_1^*.
\end{equation}
We can conclude to \eqref{Theta-by-E} in $\Upsilon_-$ if $2\nu<\psi(K^{-1})$, which is always satisfied for $K$ large enough since $\psi(0+)=+\infty$. 
\end{proof}
To conclude the proof of Proposition~\ref{prop:entropy-estimate-quad-p}, we still have to explain why it remains true when the non-linearities $f^n$ and $\sigma^n$ defined in \eqref{truncfsigma} are considered. This amounts to justify the validity of Lemma~\ref{lem:ControlSourceEntropy} in this case (with a constant independent on $n$). Consider thus
\begin{equation}\label{introSn}
S_n=f^n_1(a)\left(\ln(a_1^{*}a_3^{*})-\ln(a_2^{*} a_4^{*})\right)+\frac12 \sum_{\alpha,i}\frac{|\sigma^n_{\alpha,i}(a_i)|^2}{a_i^{*}},
\end{equation}
which is analogous to the quantity $S$ defined in \eqref{introS}. Actually, by our choice of truncation in \eqref{truncfsigma}, $S_n$ is proportional to $S$:
\begin{equation}\label{SnS}
S_n=|\chi_n(|a|)|^2 S\leq S,
\end{equation}
so the results follows at once.
\end{proof}

\section{An \texorpdfstring{$L^2\ln(L^2)$}{}-estimate by duality}\label{sec:L2duality}

Let $\Phi$, $\bar{\Phi}$ be defined by \eqref{EntropyBinary1}. Set
\begin{equation}\label{L2E}
\mathcal{E}_2(t):=\int_{\T^d}\sum_i|\bar{\Phi}(a_{i}(t))|^2 dx.
\end{equation}
Assume that
\begin{equation}\label{E20}
\mathtt{E}_2(0)=\E\left[\mathcal{E}_2(0)\right]<+\infty.
\end{equation}
 The main result of this section shows that the initial $L^2\ln(L^2)$-bound~\eqref{E20} is propagated in a bound
\begin{equation}\label{L2EQ_T}
\E\left[\int_0^T \mathcal{E}_2(t)ds\right]<+\infty.
\end{equation}
This quadratic estimate is obtained by duality, by considering an appropriate backward parabolic SPDE (see \eqref{eqwq}). This is an extension to the stochastic framework of the duality method developed in \cite{DesvillettesFellnerPierreVovelle07} (see also \cite{Pierre10}).

\begin{theorem}[$L^2\ln(L^2)$-estimate]\label{th:l2estim} Suppose that the coefficients $\sigma_i^\alpha$ satisfy \eqref{cancel-sigma-alpha} and \eqref{growth-sigma-alpha2}. Assume also $a_{i0}\in C^\infty({\T^d})$, $a_{i0}\geq 0$, $i=1,\ldots,4$. Let $a$ be a global-in-time regular solution to \eqref{SS} or a global-in-time solution to \eqref{SS} with the truncated non-linearities defined in \eqref{truncfsigma}-\eqref{truncationTn}. Then, under the smallness condition~\eqref{smallnoise}, $a$ satisfies the following estimate: for all $T>0$,
\begin{equation}\label{mainestim}
\E\left[\int_0^T \mathcal{E}_2(t)ds\right]\leq C e^{2C_1 T}\mathtt{E}_2(0),
\end{equation}
where the constant $C$ depends on $d$, $\min_{1\leq i\leq 4}\kappa_i$, $\max_{1\leq i\leq 4}\kappa_i$ only, and $C_1$ is the constant in introduced in Lemma~\ref{lem:ControlSourceEntropy}.
\end{theorem}

\begin{proof}[Proof of Theorem~\ref{th:l2estim}] The proof breaks into several steps. 

\paragraph{Step 1.} We start from the equation~\eqref{eqzi} and use \eqref{eq:controlSource} to obtain
\begin{equation}\label{eqzi2}
dz-\nx\cd\left(\sum_i\kappa_i\nx z_i\right)dt
=  F dt
+\sum_i \ln(a_i)\sigma_i^\alpha(a)d B_\alpha(t),
\end{equation}
where
\begin{equation}\label{defzzz}
z=\sum_i\bar{\Phi}(a_i),
\end{equation}
and where, for all $x\in {\T^d}$, $(F(x,t))$ is an adapted process, satisfying $F(x,t)\leq C_1 z(x,t)$ a.s., for all $(x,t)\in Q_T$. We rewrite \eqref{eqzi2} as
\begin{equation}\label{eqzi3}
dz-\Delta(K z)dt
=  F dt
+g^\alpha d B_\alpha(t),
\end{equation}
where
\begin{equation}\label{galpha}
g^\alpha=\sum_i (1+\ln(a_i))\sigma_i^\alpha(a),
\end{equation}
and where the coefficient $K$ is given\footnote{if $a(x,t)=0$, we set $K(x,t)=\underline{\kappa}$} as the convex combination
\begin{equation}\label{defK}
K(x,t)=\sum_i\frac{\bar{\Phi}(a_i(x,t))}{z(t,x)} \kappa_i\in\left[\min_{1\leq i\leq 4}\kappa_i,\max_{1\leq i\leq 4}\kappa_i\right].
\end{equation}
Thus $K$ is measurable in $(x,t)$ and bounded, with $K(x,t)\geq\underline{\kappa}$ for all $(x,t)\in Q_T$, $\PP$-a.s. by \eqref{positivekappa}, and for all $x\in {\T^d}$, $(K(x,t))$ is adapted.  Next, we introduce the solution $(w,(q^\alpha))$ to the backward SPDE 
\begin{equation}\label{eqwq}
dw(t)+(\tilde{K}(t)\Delta w(t)+C_1 w(t))dt=-H(t) dt+q^\alpha(t) d B_\alpha(t),
\end{equation}
with terminal condition 
\begin{equation}\label{TCw}
w(x,T)=0,\quad x\in {\T^d}.
\end{equation}
Let $\mathcal{P}$ denote the predictable $\sigma$-algebra of $\Omega\times[0,T]$. In \eqref{eqwq}, we assume that 
\begin{equation}\label{wHK}
\tilde{K},\, H\colon \Omega\times[0,T]\times {\T^d}\to\R
\end{equation}
are $\mathcal{P}\times\mathcal{B}({\T^d})$-measurable, and that $\PP$-a.s., for all $(x,t)\in Q_T$,
\begin{equation}\label{posHK}
\underline{\kappa}\leq\tilde{K}(x,t)\leq\max_{1\leq i\leq 4}\kappa_i,\quad 0\leq H(x,t),
\end{equation}
and that there exists an integer $q>2+d/2$, a constant $C_{\tilde{H},K}\geq 0$ such that $\PP$-a.s., for all $(x,t)\in Q_T$, for all multi-index $m$ of length $|m|<q$
\begin{equation}\label{regHK}
|D^m_x\tilde{K}(x,t)|+|D^m_x H(x,t)|\leq C_{\tilde{H},K}.
\end{equation}
Note in particular that there is a modulus $\gamma\colon [0,+\infty)\to[0,+\infty)$ (\textit{i.e.} a continuous and increasing $\gamma$ with $\gamma(r)=0$ if, and only if $r=0$) such that $\PP$-a.s., for all $t\in[0,T]$, for all $x,y\in {\T^d}$,
\begin{equation}\label{modulusK}
|\tilde{K}(t,x)-\tilde{K}(t,y)|\leq\gamma(|x-y|).
\end{equation}
By \cite{DuTang2012} (see Remark~3.1 and Corollary~3.4), the problem \eqref{eqwq}-\eqref{TCw} admits a solution $(w,(q^\alpha))$ in the following sense: 
\begin{enumerate}
\item $(w,(q^\alpha))$ satisfies the interior regularity
\begin{equation}\label{wqIntReg}
w\in L^2(\Omega\times(0,T),\mathcal{P},C^2(\T^d))\cap L^2(\Omega,C(\bar{Q}_T)),\quad q^\alpha\in L^2(\Omega\times(0,T),\mathcal{P},C^1(\T^d)),
\end{equation}
with the bound
\begin{equation}\label{l2qalpha}
\E\int_0^T \sum_\alpha |q^\alpha(x,t)|^2 dt<+\infty,
\end{equation}
for all $x\in\T^d$, 
\item the equation \eqref{eqwq} is satisfied point-wise, for every $x\in\T^d$,
\item $(w,(q^\alpha))$ satisfies the bound
\begin{multline}\label{Sobolevwqalpha}
\E\int_0^T\left(\|w(t)\|_{H^2(\T^d)}^2+\sum_{\alpha}\|q(t)\|_{H^1(\T^d)}^2\right)dt
+\E\left[\sup_{t\in[0,T]}\|w(t)\|_{H^1(\T^d)}^2\right]\\
\leq C_0\E\int_0^T\|H\|_{L^2(\T^d)}^2 dt,
\end{multline}
where the constant $C_0$ depends on $d$, $T$ and on the modulus $\gamma$ in \eqref{modulusK}.
\end{enumerate}
Note that, in \cite{DuTang2012} this is the Homogeneous Dirichlet problem that is considered, but the adaptation to the case of periodic boundary conditions is straightforward. We will justify in Step~2 the It\^o formula, which, based on the equations satisfied by $z$ and $w$, gives 
\begin{multline}\label{Itozw}
\E\left[\iint_{Q_T} H z dx dt\right]
=
\E\left[\int_{\T^d} z_0 w(0) dx\right]
+\E\left[\iint_{Q_T}\Delta w (K-\tilde{K})z dxdt\right]\\
+\E\left[\iint_{Q_T}  (F-C_1z)w dx dt\right]
+\E\left[\iint_{Q_T} g^\alpha q^\alpha dx dt\right].
\end{multline}
In Step~3, we show that $\PP$-a.s., $w(x,t)\geq 0$ for all $(x,t)\in Q_T$.  In Step~4, we establish the following bounds on $(w,(q^\alpha))$:
\begin{multline}\label{boundwq}
\E\left[\|w(0)\|_{H^1({\T^d})}^2\right]
+\E\int_0^T e^{2C_1 t}\left[\|w(t)\|_{H^2({\T^d})}^2
+\sum_\alpha\|q^\alpha\|_{H^1({\T^d})}^2\right]dt\\
\leq 
R_1\E\int_0^T e^{2C_1 t}\|H(t)\|_{L^2({\T^d})}^2 dt,
\end{multline}
where the constant $R_1$ depends on $d$, $\underline{\kappa}$ and $\max_{1\leq i\leq 4}\kappa_i$ only. In the last, fifth step, we consider the limit of \eqref{Itozw} when $\tilde{K}=K_\eps$ is a regularization of $K$, and conclude our argument.

\paragraph{Step 2.} We need to justify the It\^o formula for the product $\dual{z(t)}{w(t)}_{L^2({\T^d})}$: at least three possible approaches seem possible in our situation:
\begin{enumerate}
\item use the It\^o formula for the square of the $L^2$-norm (see \cite{KrylovRozovskii1979} for instance) of $w$, $z$, $w+z$,
\item consider the equations at fixed $x$ (this is possible since we work with regular enough solutions), use the It\^o formula for real-valued processes, and integrate the result over $\T^d$,
\item use a spectral decomposition of the processes, apply the It\^o formula for real-valued processes, and gather the results.
\end{enumerate}
This is the last method that we will employ since, even if regularity of solutions is available in our case, it is less demanding from that point of view than the second approach (and as explained in Remark~\ref{rk:DirBC}, solutions with less regularity may have to be considered if different boundary conditions are assumed). So let $(v_n)$ denote the Fourier basis of $L^2(\T^d)$: $-\Delta v_n=\lambda_n v_n$, $\lambda_n=4\pi^2|n|^2$, for all $n\in\Z^d$. We consider the spectral decomposition
\begin{equation}\label{spectralwqz}
w(t)=\sum_{n\in\Z^d}\hat{w}^n(t) v_n,\quad q^\alpha(t)=\sum_{n\in\Z^d}\hat{q}_n^\alpha(t) v_n\quad z(t)=\sum_{n\in\Z^d}\hat{z}^n(t) v_n.
\end{equation}
The integrability and regularity properties of $w$, $q^\alpha$ and $z$ ensure that the series in \eqref{spectralwqz} have at least the following convergence properties:
\begin{equation}\label{CVspectralwqz}
\E\int_0^T\sum_{n\in\Z^d}\lambda_n^2(|\hat{w}^n(t)|^2+|\hat{z}^n(t)|^2)<+\infty,\quad \E\int_0^T\sum_{\substack{n\in\Z^d \\ \alpha\geq 1}}\lambda_n|\hat{q}_n^\alpha(t)|^2<+\infty.
\end{equation}
The solutions to \eqref{eqzi3} and \eqref{eqwq} are weak solutions (\textit{cf.} the terminology in \cite[Definition~2.1]{DuTang2012} for instance), so we can test them against the spectral element $v_n$ to obtain the following set of equations, for $n\in\Z^d$:
\begin{equation}\label{eqzwn}
d\hat{z}^n(t)=f_n(t)dt+g_n^\alpha(t) d B_\alpha(t),\quad d\hat{w}^n(t)=\tilde{f}_n(t)dt+\hat{q}_n^\alpha(t) d B_\alpha(t),
\end{equation}
where
\begin{equation}
f_n(t)=\dual{\Delta(K(t)z(t))+F(t)}{v_n}_{L^2({\T^d})},\quad g_n^\alpha(t)=\dual{g^\alpha(t)}{v_n}_{L^2({\T^d})},
\end{equation}
and
\begin{equation}
\tilde{f}_n(t)=\dual{-\tilde{K}(t)\Delta w(t)-H(t)-C_1w(t)}{v_n}_{L^2({\T^d})}.
\end{equation}
By \eqref{eqzwn}, the standard It\^o formula and the terminal condition~\eqref{TCw}, we obtain
\begin{equation}\label{eqzwnIto}
0=\E\left[\hat{z}^n(0)\hat{w}^n(0)\right]
+\int_0^T\E\left[ f_n(t)\hat{w}^n(t)+\tilde{f}_n(t)\hat{z}^n(t)+g_n^\alpha(t)\hat{q}_n^\alpha(t)\right]dt .
\end{equation}
We then sum \eqref{eqzwnIto} for $n\in\Z^d$ bounded by $N$: $|n|\leq N$. We use the property
\begin{equation}\label{ippH2}
-\dual{\Delta u}{v_n}_{L^2({\T^d})}=\lambda_n\dual{u}{v_n}_{L^2({\T^d})},\quad u\in H^2({\T^d}),
\end{equation}
to obtain in particular a term
\begin{equation}\label{finiteN1}
\E\int_0^T \sum_{|n|\leq N} \left[\dual{K(t)z(t)}{v_n}_{L^2({\T^d})}\dual{\Delta w(t)}{v_n}_{L^2({\T^d})}
-\dual{\tilde{K}(t)\Delta w(t)}{v_n}_{L^2({\T^d})}\dual{z(t)}{v_n}_{L^2({\T^d})}
\right]dt.
\end{equation}
Since all the terms $Kz$, $\Delta w$, $\tilde{K}\Delta w$, $z$ involved in \eqref{finiteN1} belong to $L^2(\Omega\times(0,T)\times {\T^d})$, we have the convergence when $N\to+\infty$ of the quantity \eqref{finiteN1} to the term
\begin{equation}
\E\left[\iint_{Q_T}\Delta w (K-\tilde{K})z dxdt\right].
\end{equation}
The convergence of the other terms in the sum over $|n|\leq N$ of \eqref{eqzwnIto} is similar: we obtain \eqref{Itozw} in the limit.

\paragraph{Step 3.} We claim here that $\PP$-a.s., for all $(x,t)\in Q_T$, $w(x,t)\geq 0$. This can be proved by justification of the It\^o formula for the quantity $w\mapsto\|w^-\|_{L^2({\T^d})}^2$, where $w^-=\max(-w,0)$ is the negative part of $w$. The result also follows directly from \cite[Theorem~5.1]{DuTang2012} and \eqref{TCw}, \eqref{posHK}.

\paragraph{Step 4.} Bounds on $(w,(q^\alpha))$. The basic principle to obtain appropriate estimates on $w$ and $q^\alpha$ (see \eqref{boundwq}) is to ``multiply'' the equation \eqref{eqwq} by $\Delta w$. In the deterministic case, we use an integration by parts and the terminal condition \eqref{TCw} to obtain
\begin{multline}\label{maxregt}
\int_0^T \dual{\partial_t w}{\Delta w}_{L^2({\T^d})}dt=\\
-\int_0^T \dual{\partial_t\nabla w}{\nabla w}_{L^2({\T^d})}dt=
-\int_0^T\frac{d\;}{dt}\frac12\|\nabla w(t)\|_{L^2({\T^d})}^2 dt=\frac12\|\nabla w(0)\|_{L^2({\T^d})}^2.
\end{multline}
The analogous result for the solution to \eqref{eqwq} is more delicate to justify as $\partial_t w$ has no proper sense. We use again, as in \textbf{Step~2.}, a spectral decomposition to justify our computations. We start from the equation~\eqref{eqzwn} for $\hat{w}_n(t)$, use the It\^o formula and the terminal condition~\eqref{TCw} to compute the evolution of the square of $t\mapsto e^{C_1t} \hat{w}_n(t)$. After taking the expectation of the result, we obtain
\begin{multline}\label{squarewn}
\frac12\E\left[|\hat{w}_n(0)|^2\right]
+\E\int_0^T e^{2C_1 t}\left[\dual{-\tilde{K}(t)\Delta w(t)-H(t)}{v_n}_{L^2({\T^d})}\hat{w}_n(t) 
+\frac12\sum_\alpha|\hat{q}_n^\alpha|^2\right]dt\\
=0.
\end{multline}
We multiply \eqref{squarewn} by $\lambda_n$, sum the result over $|n|\leq N$, use \eqref{ippH2}, and the identity
\begin{equation}\label{H1spectral}
\|\nabla u\|_{L^2({\T^d})}=\sum_{n\geq 1}\lambda_n\left|\dual{u}{v_n}_{L^2({\T^d})}\right|^2,\quad u\in H^1({\T^d}),
\end{equation}
to obtain in the limit $N\to+\infty$ the following estimate:
\begin{multline}\label{MaxRegw}
\frac12\E\left[\|\nabla w(0)\|_{L^2({\T^d})}^2\right]
+\E\int_0^T e^{2C_1 t}
\left[
\dual{\tilde{K}(t)\Delta w(t)}{\Delta w(t)}_{L^2({\T^d})}
+\frac12\sum_\alpha\|\nabla q^\alpha\|_{L^2({\T^d})}^2\right]dt\\
=
-\E\int_0^T e^{2C_1 t}
\dual{H(t)}{\Delta w(t)}_{L^2({\T^d})}
dt.
\end{multline}
Using the bounds on $\tilde{K}$ in \eqref{posHK} and the Cauchy-Schwarz inequality, we deduce from \eqref{MaxRegw} that
\begin{multline}\label{MaxRegw2}
\frac12\E\left[\|\nabla w(0)\|_{L^2({\T^d})}^2\right]
+\E\int_0^T e^{2C_1 t}\left[\frac{\underline{\kappa}}{2}\|\Delta w(t)\|_{L^2({\T^d})}^2
+\frac12\sum_\alpha\|\nabla q^\alpha\|_{L^2({\T^d})}^2\right]dt\\
\leq 
\max_{1\leq i\leq 4}(\kappa_i)\E\int_0^T e^{2C_1 t}\|H(t)\|_{L^2({\T^d})}^2 dt.
\end{multline}
Since
\begin{equation}\label{L2H1spectral}
\|u\|_{H^1({\T^d})}^2\leq C(d)\|\nabla u\|_{L^2({\T^d})}^2,\quad \|u\|_{H^2({\T^d})}^2\leq C(d)\|\Delta u\|_{L^2({\T^d})}^2,
\quad u\in H^2({\T^d}),
\end{equation}
where $C(d)$ is a constant depending on $d$, \eqref{MaxRegw2} gives \eqref{boundwq} as desired.

\paragraph{Step 5.} Let us conclude the proof of  \eqref{mainestim}. We use the fact that $F\leq C_1 z$ (see \eqref{defzzz}-\eqref{eqzi3}) and $w\geq 0$, the Cauchy-Schwarz inequality and the bounds \eqref{boundwq} to deduce from \eqref{Itozw} that
\begin{multline}\label{step5-1}
\left(\E\left[\iint_{Q_T} H z dx dt\right]\right)^2\\
\leq 
R_2\left\{
\E\left[\|z_0\|_{L^2({\T^d})}^2\right]
+\E\left[\iint_{Q_T} e^{-2C_1 t}(K-\tilde{K})^2 z^2 dxdt\right]\right.\\
\left.+\E\left[\iint_{Q_T} e^{-2C_1 t} \sum_\alpha |g^\alpha|^2 dx dt\right]
\right\}
\left\{\E\int_0^T e^{2C_1 t}\|H(t)\|_{L^2({\T^d})}^2 dt\right\},
\end{multline}
where the constant $R_2$ depends on $d$, $\underline{\kappa}$ and $\max_{1\leq i\leq 4}\kappa_i$ only. We apply \eqref{step5-1} with $\tilde{K}=\tilde{K}_\eps$, where
\begin{equation}\label{defKeps}
\tilde{K}_\eps(x,t)=\frac{z(x,t)}{z(x,t)+\eps}K(x,t)=\frac{1}{z(x,t)+\eps}\sum_i \kappa_i\bar{\Phi}(a_i(x,t)).
\end{equation}
This coefficient $\tilde{K}$ has the desired regularity properties by regularity of $z$, and satisfies the bound
\begin{equation}
0\leq (K-\tilde{K}_\eps)z\leq\eps,
\end{equation}
so, taking the limit $\eps\to 0$ in \eqref{step5-1} where $\tilde{K}=\tilde{K}_\eps$, gives us
\begin{multline}\label{step5-2}
\left(\E\left[\iint_{Q_T} H z dx dt\right]\right)^2\\
\leq 
R_2\left\{
\E\left[\|z_0\|_{L^2({\T^d})}^2\right]+\E\left[\iint_{Q_T} e^{-2C_1 t}\sum_\alpha |g^\alpha|^2 dx dt\right]
\right\}
\left\{\E\int_0^T e^{2C_1 t}\|H(t)\|_{L^2({\T^d})}^2 dt\right\}.
\end{multline}
We replace $H$ with $t\mapsto e^{C_1 t}H(t)$ and, by an argument of density relax the hypothesis of regularity of $H$ to obtain
\begin{multline}\label{step5-3}
\left(\E\left[\iint_{Q_T} H \hat{z} dx dt\right]\right)^2\\
\leq 
R_2\left\{
\E\left[\|z_0\|_{L^2({\T^d})}^2\right]+\E\left[\iint_{Q_T} e^{-2C_1 t} \sum_\alpha|g^\alpha|^2 dx dt\right]
\right\}
\left\{\E\int_0^T \|H(t)\|_{L^2({\T^d})}^2 dt\right\},
\end{multline}
for all non-negative $H\in L^2(\Omega\times(0,T),\mathcal{P};L^2({\T^d}))$, where $\hat{z}(t)=e^{-C_1 t}z(t)$. At last, we estimate the term 
\begin{equation}\label{gggalpha}
\E\left[\iint_{Q_T} e^{-2C_1 t} |g^\alpha|^2 dx dt\right].
\end{equation}
Recall that $g^\alpha$ is defined in \eqref{galpha}. By \eqref{growth-sigma-alpha2}, and \eqref{aiaj} we have
\begin{equation}
\sum_\alpha|g^\alpha|^2\leq R_3\nu z^2
\end{equation}
where $R_3$ is a numerical constant, and so 
\begin{equation}\label{step5-4}
\E\left[\iint_{Q_T} e^{-2C_1 t} |g^\alpha|^2 dx dt\right]\leq R_3\nu \E\left[\iint_{Q_T} |\hat{z}|^2 dx dt\right].
\end{equation}
We insert \eqref{step5-4} in \eqref{step5-3} and invoke \eqref{smallnoise} to ensure that $R_2R_3\nu\leq\frac12$ (note well that the constant $R_2R_3$ does not depend on time) to conclude that 
\begin{equation}\label{step5-5}
\E\left[\iint_{Q_T} |\hat{z}|^2 dx dt\right]\leq R_4\E\left[\|z_0\|_{L^2({\T^d})}^2\right],
\end{equation}
which yields \eqref{mainestim}.
\end{proof}


\section{\texorpdfstring{$L^\infty$}{}-estimates}\label{sec:LinftyRD}

\subsection{\texorpdfstring{$L^\infty$}{}-estimate via truncation of the entropy}\label{subsec:LinftyRDproof}

Various approaches to $L^\infty$-estimates for (deterministic) quadratic reaction-diffusion systems exist (see \cite{CaputoVasseur2009,GoudonVasseur10,CaputoGoudonVasseur2019,Souplet2018,FellnerMorganTang2020}). Here we follow the same approach \`a la De Giorgi as in \cite{GoudonVasseur10}, using also some of the elements in \cite{LeocataVovelle2023}.

\begin{theorem}[$L^\infty$-estimate]\label{th:Linfty-estimate} Let
\begin{equation}\label{Theta}
\Theta(u)=\ln(1+\ln(1+u)),\quad u\geq 0.
\end{equation}
Suppose that the coefficients $\sigma_i^\alpha$ satisfy \eqref{cancel-sigma-alpha} and \eqref{growth-sigma-alpha2}. Assume also $a_{i0}\in C^\infty({\T^d})$, $a_{i0}\geq 0$, $i=1,\ldots,4$. Let $a$ be a global-in-time regular solution to \eqref{SS} or a global-in-time solution to \eqref{SS} with the truncated non-linearities defined in \eqref{truncfsigma}-\eqref{truncationTn}. 
Assume also that the space dimension $d$ is $d=1$ or $d=2$ and that the strength of noise is small in the sense of \eqref{smallnoise}. Then $a$ satisfies the following estimate: 
\begin{equation}\label{eq:Linftyestimate}
\E\left[\Theta\left(\sup_{1\leq i\leq 4}\|a_i\|_{L^\infty(Q_T)}\right)\right]\leq C_\infty,
\end{equation}
where the constant $C_\infty$ depends on $d$, $\nu$, $T$, $\min_{1\leq i\leq 4}\kappa_i$, $\max_{1\leq i\leq 4}\kappa_i$ and $\sup_{1\leq i\leq 4}\|a_{i0}\|_{L^\infty({\T^d})}$.
\end{theorem}

\begin{proof}[Proof of Theorem~\ref{th:Linfty-estimate}] We will consider only the case $d=2$.
For $a,\xi\geq 0$, set
\begin{equation}\label{acutxi}
\axi=(a-\xi)^+,\quad a^{\xi*}=1+\axi,
\end{equation}
and
\begin{equation}\label{Phicutxi}
\Phi(a;\xi)=\Phi(\axi)=(1+\axi)\ln(1+\axi)-\axi=a^{\xi*}\ln(a^{\xi*})-a^{\xi*}+1.
\end{equation}

\paragraph{Evolution of the entropy.} Note that
\begin{equation}\label{DiffPhixi}
\frac{\partial\;}{\partial a}\Phi(a;\xi)=\ln(a^{\xi*})=\ln(a^{\xi*})\mathbf{1}_{a>\xi},\quad \frac{\partial^2\;}{\partial a^2}\Phi(a;\xi)=\frac{\mathbf{1}_{a>\xi}}{a^{\xi*}},
\end{equation}
so, similarly to \eqref{entropy1p}, we have
\begin{equation}\label{entropy1xi}
\mathcal{E}(t;\xi)+\mathcal{D}(t;\xi)\leq\mathcal{S}(t;\xi)+\mathcal{M}(t;\xi)+\mathcal{E}(0;\xi),
\end{equation}
where
\begin{equation}\label{anovED}
\mathcal{E}(t;\xi)=\sum_{i=1}^4\int_{\T^d}\Phi(a_i^\xi(t))dx,\quad \mathcal{D}(t;\xi)=\E\sum_{i=1}^4\int_0^t\int_{\T^d} \frac{|\nabla a_i^\xi|^2}{a_i^{\xi*}}dxds,
\end{equation}
and
\begin{equation}\label{noavS}
\mathcal{S}(t;\xi)=\int_0^t\int_{\T^d} \left\{-(a_1a_3-a_2a_4)\left(\ln(a_1^{\xi*}a_3^{\xi*})-\ln(a_2^{\xi*} a_4^{\xi*})\right)+\frac12 \sum_{\alpha,i}\frac{|\sigma_i^\alpha(a_i)|^2}{a_i^{\xi*}}\mathbf{1}_{a_i>\xi}
\right\}dx ds.
\end{equation}
The martingale term is
\begin{equation}\label{noavM}
\mathcal{M}(t;\xi)=\int_0^t\int_{\T^d} \sum_{\alpha,i}\sigma_i^\alpha(a_i)\ln(a^{\xi*}_i) dx d B_\alpha(s).
\end{equation}
As in \eqref{entropyestimatep-bis}, we introduce the quantity
\begin{equation}\label{calUtxi}
\mathcal{U}(t;\xi)=\sup_{0\leq s\leq t}\mathcal{E}(s;\xi)+\mathcal{D}(t;\xi).
\end{equation}
\paragraph{Auxiliary functional.} Set
\begin{equation}\label{goodPsi}
\Psi(a)=\int_0^a \sqrt{\frac{\Phi(s)}{s^*}}ds=\int_0^a \sqrt{\frac{(1+s)\ln(1+s)-s}{1+s}}ds
\end{equation}
and
\begin{equation}\label{calUtxipsi}
\mathcal{U}_\psi(t;\xi)=\left[\int_0^t\int_{\T^d} \sum_i |\Psi(a_i^\xi)|^2 dxds\right]^{1/2}.
\end{equation}
We have
\begin{equation}\label{nablaPsi}
\nabla\Psi(a)=\sqrt{\Phi(a)}\frac{\nabla a}{\sqrt{1+a}},
\end{equation}
so, by Sobolev's embedding (recall that the space dimension is $d=2$ here) and Cauchy-Schwarz' inequality,
\begin{equation}\label{PsibynablaPsi}
\|\Psi(a)\|_{L^2({\T^d})}\lesssim \|\nabla\Psi(a)\|_{L^1({\T^d})}\lesssim\|\Phi(a)\|_{L^1({\T^d})}^\frac12\|(1+a)^{-\frac12}\nabla a\|_{L^2({\T^d})},
\end{equation}
which yields
\begin{equation}
\mathcal{U}_\psi(t;\xi)^2\lesssim\int_0^t \mathcal{E}(s;\xi)\sum_i\left\|(1+a_i^\xi)^{-\frac12}\nabla a_i^\xi\right\|_{L^2({\T^d})}^2 ds\lesssim\left[\sup_{s\in[0,t]}\mathcal{E}(s;\xi)\right]\mathcal{D}(t;\xi),
\end{equation}
and thus
\begin{equation}\label{nablaPsiviaEntropy}
\mathcal{U}_\psi(t;\xi)\lesssim \mathcal{U}(t;\xi).
\end{equation}
As in \cite[Lemma 3.2]{GoudonVasseur10}, we will use the fact that
\begin{equation}\label{gPsi}
g(a)\ln(a^*)\lesssim|\Psi(a)|^2,\quad g(a):=a^3\mathbf{1}_{0\leq a\leq 1}+a^2\mathbf{1}_{a>1}.
\end{equation}
\paragraph{Bound on the source term.} We will proceed as in the previous section~\ref{sec:entropyestimate}, taking into account the additional effects of the truncation at level $\xi$. In particular, our aim is not a bound like \eqref{eq:controlSource}, but instead, using \eqref{nablaPsiviaEntropy}-\eqref{gPsi}, a bound of $\mathcal{S}(t;\xi)$ by a power of $\mathcal{U}(t;\zeta)$ (or more exactly, a power of $\mathcal{U}_\psi(t;\zeta)$), for a given truncation level $\zeta<\xi$. We begin as in Section~\ref{sec:entropyestimate} however, and use first the growth condition \eqref{growth-sigma-alpha2}, to get
\begin{multline}\label{boundSsto1}
\mathcal{S}(t;\xi)\leq \int_0^t\int_{\T^d} \left\{-(a_1^{\xi*}a_3^{\xi*}-a_2^{\xi*}a_4^{\xi*})\left(\ln(a_1^{\xi*}a_3^{\xi*})-\ln(a_2^{\xi*} a_4^{\xi*})\right)\right\}dxds\\
+\nu\int_0^t\int_{\T^d}  \sum_{i}\frac{a_1^{\xi*}a_3^{\xi*}+a_2^{\xi*}a_4^{\xi*}}{a_i^{\xi*}}\mathbf{1}_{a_i>\xi}dx ds+\varphi(t;\xi),
\end{multline}
where
\begin{multline}\label{boundSsto1e}
\varphi(t;\xi)= \int_0^t\int_{\T^d} \left\{(a_1^{\xi*}a_3^{\xi*}-a_2^{\xi*}a_4^{\xi*})-(a_1a_3-a_2a_4)\right\}\left(\ln(a_1^{\xi*}a_3^{\xi*})-\ln(a_2^{\xi*} a_4^{\xi*})\right)dx ds\\
+\nu \int_0^t\int_{\T^d} \sum_{i}\frac{(a_1a_3+a_2a_4)-(a_1^{\xi*}a_3^{\xi*}+a_2^{\xi*}a_4^{\xi*})}{a_i^{\xi*}}\mathbf{1}_{a_i>\xi}
dx ds. 
\end{multline}
We have 
\begin{equation}\label{preR12}
a_ia_j-a_i^{\xi*}a_j^{\xi*}=a_ia_j-a_i^{\xi}a_j^{\xi}-1-(a_i^\xi+a_j^\xi).
\end{equation}
The term $a_ia_j-a_i^{\xi}a_j^{\xi}$ is non-positive, so
\begin{equation}\label{R12a}
-1-(a_i^\xi+a_j^\xi)\leq a_ia_j-a_i^{\xi*}a_j^{\xi*}.
\end{equation}
From \eqref{preR12}, we also obtain the bound from below $a_ia_j-a_i^{\xi*}a_j^{\xi*}\leq a_ia_j-a_i^{\xi}a_j^{\xi}$, and since $a_i\leq a_i^\xi+\xi$, we get
\begin{equation}\label{R12}
-1-(a_i^\xi+a_j^\xi)\leq a_ia_j-a_i^{\xi*}a_j^{\xi*}\leq\xi^2+\xi(a_i^\xi+a_j^\xi),
\end{equation}
which implies in particular
\begin{equation}\label{R12plus}
|a_ia_j-a_i^{\xi*}a_j^{\xi*}|\leq\xi^2+\xi(a_i^\xi+a_j^\xi),
\end{equation}
if 
\begin{equation}\label{largexi}
\xi\geq 1.
\end{equation}
Assuming \eqref{largexi} thus, we obtain the first estimate
\begin{multline}\label{boundSsto1e2}
\varphi(t;\xi)\leq 2\int_0^t\int_{\T^d} \left(\xi^2+\xi\sum_i a_i^\xi\right)\left(\ln(a_1^{\xi*}a_3^{\xi*})+\ln(a_2^{\xi*} a_4^{\xi*})\right)dx ds\\
+\nu \int_0^t\int_{\T^d} \sum_{i}\frac{\xi^2+\xi\sum_j a_j^\xi}{a_i^{\xi*}}\mathbf{1}_{a_i>\xi}
dx ds.
\end{multline}
Next, we will use the following inequalities (similar to \eqref{aiaj}):
\begin{equation}\label{aiaj1}
a_j^\xi\ln(a_i^{\xi*})\leq a_i^\xi\ln(a_i^{\xi*})+a_j^\xi\ln(a_j^{\xi*}),
\quad a_j^{\xi}\mathbf{1}_{a_i>\xi}\leq a_i^{\xi}\mathbf{1}_{a_i>\xi}+a_j^{\xi}\mathbf{1}_{a_j>\xi}.
\end{equation}
In the second term in \eqref{boundSsto1e2} we also use the elementary bound 
$(a_i^{\xi*})^{-1}\leq 1$ to get
\begin{equation}\label{boundSsto1e3}
\varphi(t;\xi)\lesssim \int_0^t\int_{\T^d} \sum_i (\xi^2+\xi a_i^{\xi})(\ln(a_i^{\xi*})+\mathbf{1}_{a_i>\xi})dx ds.
\end{equation}
Let then $\zeta\in (0,\xi)$. We have $a^\xi\leq a^\zeta$ and 
\begin{equation}\label{bound1axi}
\mathbf{1}_{a>\xi}\leq\left(\frac{a^\zeta}{\xi-\zeta}\right)^{\alpha+1}\mathbf{1}_{a^\zeta\leq 1}
+\left(\frac{a^\zeta}{\xi-\zeta}\right)^\alpha\mathbf{1}_{1<a^\zeta},
\end{equation}
for some arbitrary exponent $\alpha\geq 0$. Taking successively $\alpha=2,1$, and using the fact that $\ln(a^{\xi*})=0$ if $a\leq\xi$, we deduce from \eqref{gPsi}-\eqref{bound1axi} that
\begin{equation}\label{bound1axiAll}
\ln(a^{\xi*})\leq W_2(\xi,\zeta)|\Psi(a^\zeta)|^2,\quad
a^\xi\ln(a^{\xi*})\leq W_1(\xi,\zeta)|\Psi(a^\zeta)|^2
\end{equation}
where
\begin{equation}\label{Wxizeta}
W_\alpha(\xi,\zeta)=\frac{1}{(\xi-\zeta)^{\alpha+1}}+\frac{1}{(\xi-\zeta)^\alpha}.
\end{equation}
We also note
\begin{equation}
a>\xi\Rightarrow a^{\zeta*}\geq 1+(\xi-\zeta)\Rightarrow \frac{\ln(a^{\zeta*})}{\ln(1+(\xi-\zeta))}\geq 1,
\end{equation}
so, similarly to \eqref{bound1axiAll}, we have
\begin{equation}\label{bound111}
\mathbf{1}_{a>\xi}\leq \frac{W_2(\xi,\zeta)}{\ln(1+(\xi-\zeta))}|\Psi(a^\zeta)|^2,\quad
a^\xi \mathbf{1}_{a>\xi}\leq \frac{W_1(\xi,\zeta)}{\ln(1+(\xi-\zeta))}|\Psi(a^\zeta)|^2.
\end{equation}
Finally, we deduce from \eqref{boundSsto1e3}, \eqref{bound1axiAll}, \eqref{bound111} and \eqref{nablaPsiviaEntropy} that 
\begin{equation}\label{boundSsto1e4}
\varphi(t;\xi)\lesssim \frac{1}{1\wedge \ln(1+(\xi-\zeta))}\left[\xi^2 W_2(\xi,\zeta)+\xi W_1(\xi,\zeta)\right]\mathcal{U}_\psi(t;\xi)^2,
\end{equation}
and obtain therefore a bound from above on the last term in the right-hand side of \eqref{boundSsto1}. The other two terms in the right-hand side of \eqref{boundSsto1} can be gathered to form a quantity very similar to the the function $\Theta(a)$ in \eqref{integrandSS}. The procedure followed to prove \eqref{Theta-by-E} can be adapted to establish 
\begin{multline}\label{boundSsto2}
\int_0^t\int_{\T^d} \left\{-(a_1^{\xi*}a_3^{\xi*}-a_2^{\xi*}a_4^{\xi*})\left(\ln(a_1^{\xi*}a_3^{\xi*})-\ln(a_2^{\xi*} a_4^{\xi*})\right)\right\}dxds\\
+\nu\int_0^t\int_{\T^d}  \sum_{i}\frac{a_1^{\xi*}a_3^{\xi*}+a_2^{\xi*}a_4^{\xi*}}{a_i^{\xi*}}\mathbf{1}_{a_i>\xi}dx ds
\lesssim \int_0^t\int_{\T^d} \sum_{i,j}a^{\xi*}_j \mathbf{1}_{a_i>\xi}dxds.
\end{multline}
By \eqref{aiaj1}, \eqref{bound111}, and \eqref{boundSsto1e4}, we obtain
\begin{equation}\label{boundSsto3}
\mathcal{S}(t;\xi)\lesssim \frac{1}{1\wedge \ln(1+(\xi-\zeta))}\left[\xi^2 W_2(\xi,\zeta)+\xi W_1(\xi,\zeta)\right]\mathcal{U}_\psi(t;\xi)^2,
\end{equation}
assuming \eqref{largexi}. We report the estimate \eqref{boundSsto3} in \eqref{entropy1xi}. Assume that 
\begin{equation}\label{largexi2}
1\vee \max_{1\leq i\leq 4}\|a_{i0}\|_{L^\infty({\T^d})}\leq\xi.
\end{equation}
Then $\mathcal{E}(0;\xi)=0$ and we obtain the a.s. inequalities
\begin{equation}\label{CLBoundS}
\mathcal{U}(t;\xi)
\lesssim \frac{1}{1\wedge \ln(1+(\xi-\zeta))}\left[\xi^2 W_2(\xi,\zeta)+\xi W_1(\xi,\zeta)\right]\mathcal{U}_\psi(t;\zeta)^2+\mathcal{M}(t;\xi)^*,
\end{equation}
where
\begin{equation}\label{Mstar}
\mathcal{M}(t;\xi)^*=\sup_{s\in[0,t]}|\mathcal{M}(s;\xi)|.
\end{equation}
\paragraph{Bound on the martingale term.} The quadratic variation of the martingale $\mathcal{M}(t;\xi)$ in \eqref{noavM} is
\begin{equation}\label{noavMQ}
\dual{\mathcal{M}(\xi)}{\mathcal{M}(\xi)}_t=\int_0^t\sum_\alpha\left|\int_{\T^d} \sum_{i}\sigma_i^\alpha(a)\ln(a^{\xi*}_i) dx\right|^2 ds.
\end{equation}
We proceed as in \eqref{QvarMt2} (simply replace $\ln(a^{*}_j)$ by $\ln(a^{\xi*}_j)$) to obtain
\begin{equation}
\dual{\mathcal{M}(\xi)}{\mathcal{M}(\xi)}_t
\lesssim \int_0^t\left|\int_{\T^d} \sum_{i,j}a_i\ln(a^{\xi*}_j)dx\right|^2 ds.\label{noavMQ2}
\end{equation}
The bounds $a_i\leq a_i^\xi+\xi$ and \eqref{aiaj1} show that 
\begin{equation}\label{noavMQ3}
\int_{\T^d} \sum_{i,j}a_i\ln(a^{\xi*}_j)dx\lesssim \sum_i\int_{\T^d} (\xi+a_i^\xi)\ln(a^{\xi*}_i)dx.
\end{equation}
We can estimate the right-hand side of \eqref{noavMQ3} in two different ways. First, using \eqref{bound111}, we get, for $\zeta<\xi$,
\begin{equation}\label{noavMQ4}
\int_{\T^d} \sum_{i,j}a_i\ln(a^{\xi*}_j)dx\lesssim \frac{\xi W_2(\xi,\zeta)+W_1(\xi,\zeta)}{\ln(1+(\xi-\zeta))}\sum_i\int_{\T^d} |\Psi(a_i^\zeta)|^2 dx.
\end{equation}
We also have $a_i^\xi\ln(a^{\xi*}_i)\leq 2\Phi(a_i^\xi)\leq 2\Phi(a_i^\zeta)$ and 
\begin{equation}\label{lnbyPhi}
\xi\ln(a^{\xi*}_i)\leq\xi\frac{a^\zeta}{\xi-\zeta}\ln(a^{\zeta*}_i)\leq 2\frac{\xi}{\xi-\zeta}\Phi(a_i^\zeta),
\end{equation}
so
\begin{equation}\label{noavMQ5}
\sup_{0\leq s\leq t}\int_{\T^d} \sum_{i,j}a_i\ln(a^{\xi*}_j)dx\lesssim \frac{\xi}{\xi-\zeta}\sup_{0\leq s\leq t}\mathcal{E}(s;\zeta)\lesssim  \frac{\xi}{\xi-\zeta}\mathcal{U}(t;\zeta).
\end{equation}
Using both \eqref{noavMQ4} and \eqref{noavMQ5} in \eqref{noavMQ2}, we obtain
\begin{equation}\label{noavMQ6}
\dual{\mathcal{M}(\xi)}{\mathcal{M}(\xi)}_t\lesssim \left(\frac{\xi}{\xi-\zeta}\frac{\xi W_2(\xi,\zeta)+W_1(\xi,\zeta)}{\ln(1+(\xi-\zeta))}\right)\mathcal{U}(t;\zeta)\mathcal{U}_\Psi(t;\zeta)^2,
\end{equation}
and from \eqref{nablaPsiviaEntropy}, we deduce that
\begin{equation}\label{noavMQ7}
\dual{\mathcal{M}(\xi)}{\mathcal{M}(\xi)}_t\lesssim \left(\frac{\xi}{\xi-\zeta}\frac{\xi W_2(\xi,\zeta)+W_1(\xi,\zeta)}{\ln(1+(\xi-\zeta))}\right)\mathcal{U}(t;\zeta)^3.
\end{equation}

\paragraph{Recursive estimates.} Let $\xi$ be a fixed threshold satisfying
\begin{equation}\label{largebarxi}
4\left(1\vee \max_{1\leq i\leq 4}\|a_{i0}\|_{L^\infty({\T^d})}\right)\leq\xi.
\end{equation}
Let also $\underline{\rho}$ be a given exponent satisfying
\begin{equation}\label{goodurho}
1<\underline{\rho}<\frac{1+\rho}{2}=\frac32.
\end{equation}
Let $\delta\in(0,1)$ be a given parameter (which will eventually depend on $\xi$, see \eqref{deltaxi}). We set $\xi_k=(1-2^{-k-1})\xi$, $k\in\N$ and we examine the occurrence of the bound
\begin{equation}\label{boundUk}
\mathcal{U}_k\leq\delta^{\underline{\rho}^k},\quad\mathcal{U}_k:=\mathcal{U}(T;\xi_k).
\end{equation}
Assume first that \eqref{boundUk} is satisfied at rank $k$. Then, using \eqref{nablaPsiviaEntropy} and the inequality \eqref{CLBoundS} with $\zeta=\xi_k$ and $\xi=\xi_{k+1}$, we observe that
\begin{equation}\label{boundUkp1}
\mathcal{U}_{k+1}\leq 
C_1\frac{1}{1\wedge\ln(1+2^{-k-2}\xi)}W_k
\delta^{\rho\underline{\rho}^k}+C_1\mathcal{M}(t;\xi_{k+1})^*,
\end{equation}
where
\begin{equation}\label{Wk}
W_k=\xi_{k+1}^2 W_2(\xi_{k+1},\xi_k)+\xi_{k+1} W_1(\xi_{k+1},\xi_k)
\lesssim \frac{8^{k}}{\xi}+4^{k},
\end{equation}
so \eqref{boundUk} will be satisfied at rank $k+1$ if
\begin{equation}\label{OK1}
C_1\frac{W_k}{1\wedge\ln(1+2^{-k-2}\xi)}
\delta^{(\rho-\underline{\rho})\underline{\rho}^k}
\leq\frac12,
\end{equation}
and
\begin{equation}\label{OK2}
\mathcal{M}(T;\xi_{k+1})^*\leq \frac12\delta^{\underline{\rho}^{k+1}}.
\end{equation}
Let $E_{k+1}$ denote the event \eqref{OK2} and $H_k$ denote the event \eqref{boundUk}. Our aim is to evaluate the probability $\PP(\mathbf{H})$ of the event 
\begin{equation}\label{goodEventU}
\mathbf{H}=\bigcap_{k\geq 0} H_k=\bigcap_{k\geq 0}\mathbf{H}_k,\quad \mathbf{H}_k:=\bigcap_{j=0}^k H_j.
\end{equation}
We will estimate $\sum_{k\geq 0} p_k-p_{k+1}$, where $p_k=\PP(\mathbf{H}_k)$, and then, in the next step, evaluate $p_0$.
Assume that \eqref{OK1} is satisfied for all $k\geq 0$ (we will see that an appropriate choice of parameter $\delta$ ensures this). Then $\mathbf{H}_k\cap E_{k+1}\subset \mathbf{H}_{k+1}$, so
\begin{equation}\label{decaypkU}
p_k-p_{k+1}\leq \PP(\mathbf{H}_{k})-\PP(\mathbf{H}_{k}\cap E_{k+1})=\PP(\mathbf{H}_{k}\cap E_{k+1}^c)
\end{equation}
We use the exponential martingale inequality
\begin{equation}\label{expmartingaleU}
\PP\left(M^*_\infty\geq a+b\dual{M}{M}_\infty\right)\leq e^{-2ab}
\end{equation}
with $M_t=\mathcal{M}(t\wedge T;\xi_{k+1})$ and some deterministic numbers $a=V_k$, $b=\hat{V}_k^{-1}$ that will be fixed later (see \eqref{VVkU}), to get 
\begin{equation}\label{evaluatesmallMT1U}
\PP(\mathbf{H}_k\cap E_{k+1}^c)\leq e^{-2V_k\hat{V}_k^{-1}}+\PP(B_k),
\end{equation}
where $B_k$ is the event
\begin{equation}\label{HkU}
B_k=\mathbf{H}_k\cap\left\{\frac12 \delta^{\underline{\rho}^{k+1}}\leq V_k+\hat{V}_k^{-1}\dual{\mathcal{M}(\xi_{k+1})}{\mathcal{M}(\xi_{k+1})}_T\right\}.
\end{equation}
We use the estimate \eqref{noavMQ7} on the quadratic variation of $\mathcal{M}(\xi)$ to obtain
\begin{equation}\label{noavMQ3k}
\dual{\mathcal{M}(\xi_{k+1})}{\mathcal{M}(\xi_{k+1})}_T \leq C_2 \frac{2^k W_k}{\xi\ln(1+2^{-k-2}\xi)}\mathcal{U}_k^3,
\end{equation}
and thus
\begin{equation}\label{noavMQ3k2}
\dual{\mathcal{M}(\xi_{k+1})}{\mathcal{M}(\xi_{k+1})}_T \leq C_2 \frac{2^k W_k}{\xi\ln(1+2^{-k-2}\xi)}\delta^{3\underline{\rho}^k},
\end{equation}
if $H_k$ is realized. We choose $V_k$ and $\hat{V}_k$ as follows:
\begin{equation}\label{VVkU}
V_k=\frac18\delta^{\underline{\rho}^{k+1}},\quad \hat{V}_k^{-1}=V_k\left[C_2 \frac{2^k W_k}{\xi\ln(1+2^{-k-2}\xi)}\delta^{3\underline{\rho}^k}\right]^{-1}.
\end{equation}
With this choice of the parameters, $B_k$ has probability $0$ and \eqref{evaluatesmallMT1U} yields
\begin{equation}\label{GoodDecaypk}
p_k-p_{k+1}\leq
\exp\left(
-\frac{\xi\ln(1+2^{-k-2}\xi)}{C_3 2^k W_k}\delta^{-\underline{\rho}^k(3-2\underline{\rho})}
\right).
\end{equation}
Let $K$ denote the index such that $2^{-K}\xi\approx 1$ (we are interested in large values of $\xi$, so large values of $K$ as well). For $k\leq K$, we have $W_k\lesssim 4^k$ by \eqref{Wk} and $1\lesssim\ln(1+2^{-k-2}\xi)$, so \eqref{OK1} is realized if
\begin{equation}\label{OK1sub}
C_4 4^k
\delta^{(\rho-\underline{\rho})\underline{\rho}^k}
\leq\frac12,
\end{equation}
while \eqref{GoodDecaypk} implies 
\begin{equation}\label{GoodDecaypksub}
p_k-p_{k+1}\leq
\exp\left(
-\frac{1}{C_5}4^{-k}\delta^{-\underline{\rho}^k(3-2\underline{\rho})}
\right).
\end{equation}
The map $t\mapsto 4^t \delta^{(\rho-\underline{\rho})\underline{\rho}^t}$ is non-increasing on $\R_+$ if 
$\ln(4)\leq\ln(\underline{\rho})(\rho-\underline{\rho})|\ln(\delta)|$ so \eqref{OK1sub} is satisfied if
\begin{equation}\label{smalldeltassub}
\ln(4)\leq\ln(\underline{\rho})(\rho-\underline{\rho})|\ln(\delta)|,\quad C_4\delta^{\rho-\underline{\rho}}
\leq\frac12,
\end{equation}
which is a condition of the form 
\begin{equation}\label{smalldeltassub2}
C_6\delta\leq 1.
\end{equation}
Let us examine \eqref{GoodDecaypksub} now, still when $k\leq K$. The parameter $3-2\underline{\rho}$ is positive by \eqref{goodurho}. By increasing the value of $C_6$ in \eqref{smalldeltassub2} if necessary, we can assume that 
\begin{equation}
(3-2\underline{\rho})|\ln(\delta)|\geq 2\ln(4),
\end{equation}  
in which case the map $\varphi_\flat\colon t\mapsto 4^{-t}\delta^{-(3-2\underline{\rho})\underline{\rho}^t}$ is non-decreasing and satisfies
\begin{equation}\label{goodvarphiflat}
\varphi_\flat'(t)\geq \ln(4)\varphi_\flat(t)\geq \ln(4)
\end{equation} 
It follows then from \eqref{GoodDecaypksub} that
\begin{equation}\label{sumksub}
\sum_{k\leq K}p_k-p_{k+1}\leq \int_0^K\Phi_\flat(t)dt,\quad \Phi_\flat(t):=\exp\left(
-\frac{1}{C_5}\varphi_\flat(t)
\right).
\end{equation}
By \eqref{goodvarphiflat}, we have 
\begin{equation}\label{diffPhiflat}
-\Phi_\flat'(t)\geq\frac{\ln(4)}{C_5}\Phi_\flat(t)=\frac{1}{C_7}\Phi_\flat(t),
\end{equation}
so
\begin{equation}\label{sumksub2}
\sum_{k\leq K}p_k-p_{k+1}\leq C_7\Phi_\flat(0)\leq C_7\exp\left(-C_5^{-1}\delta^{-(3-2\underline{\rho})}\right).
\end{equation}
If $k\geq K$, then $W_k\lesssim 8^k\xi^{-1}$ by \eqref{Wk} and $2^{-k}\xi\lesssim\ln(1+2^{-k-2}\xi)$, so 
\eqref{OK1} is satisfied if 
\begin{equation}\label{OK1sup}
C_4 \frac{16^k}{\xi^2}
\delta^{(\rho-\underline{\rho})\underline{\rho}^k}
\leq\frac12,
\end{equation}
whereas \eqref{GoodDecaypk} implies 
\begin{equation}\label{GoodDecaypksup}
p_k-p_{k+1}\leq
\exp\left(
-\frac{\xi^32^{-5k}}{C_3}\delta^{-\underline{\rho}^k(3-2\underline{\rho})}
\right).
\end{equation}
Since $\xi\geq 1 $ (see \eqref{largebarxi}), the first criterion \eqref{OK1sup} can be reduced to 
\begin{equation}\label{OK1sup2}
C_4 16^k
\delta^{(\rho-\underline{\rho})\underline{\rho}^k}
\leq\frac12,
\end{equation}
which follows from \eqref{smalldeltassub2}, by increasing the value of $C_6$ if necessary. We can also assume that, under \eqref{smalldeltassub2}, we have 
\begin{equation}
(3-2\underline{\rho})|\ln(\delta)|\geq 2\times 5\ln(2),
\end{equation}
in which case the same analysis as above shows that
\begin{equation}\label{sumksup}
\sum_{k\geq K}p_k-p_{k+1}\leq C_8\Phi_\sharp	(K),\quad\Phi_\sharp(t):=\exp\left(
-\frac{\xi^32^{-5t}}{C_3}\delta^{-\underline{\rho}^t(3-2\underline{\rho})}
\right).
\end{equation}
The rough estimate $\Phi_\sharp(K)\leq\Phi_\sharp	(0)\leq\exp\left(-C_3^{-1}\xi^3\right)$, which yields
\begin{equation}\label{sumksup2}
\sum_{k\geq K}p_k-p_{k+1}\leq C_8\exp\left(-C_3^{-1}\xi^3\right),
\end{equation}
will be sufficient for our purpose. Indeed, using \eqref{sumksub2} and \eqref{sumksup2}, we obtain
\begin{equation}\label{GoodEventUp0U}
\PP(\mathbf{H})\geq p_0-C_7\exp\left(-C_5^{-1}\delta^{-(3-2\underline{\rho})}\right)-C_8\exp\left(-C_3^{-1}\xi^3\right).
\end{equation}
We will take $\delta\approx (\ln(\xi))^{-\frac12}$ in \eqref{deltaxi} below, so the two last terms in \eqref{GoodEventUp0U} have a very fast decay in $\xi$. We will see below how to estimate $p_0$ (and actually will get a decay in $\xi$ which is much slower, \textit{cf.} \eqref{Tailp0}).
\paragraph{Initial estimate.} We wish to estimate the probability $p_0$ of the event 
\begin{equation}
\left\{\mathcal{U}_0=\mathcal{U}(T;\xi/2)\leq \delta\right\}.
\end{equation} 
By the Markov inequality and the inequality \eqref{CLBoundS} used with $\zeta=\xi/4$, we have
\begin{equation}\label{Markov000}
\PP\left[\mathcal{U}_0\geq \delta \right]\leq\delta^{-1}\E\left[\mathcal{U}(T;\xi/2)\right]
\lesssim
\delta^{-1}\E\left[\mathcal{U}_\psi(T;\xi/4)^2\right]+\delta^{-1}\E\left[\mathcal{M}(T;\xi/2)^*\right].
\end{equation}
In the right-hand side of \eqref{Markov000}, the term related to the martingale part can be estimated re\-la\-ti\-ve\-ly easily.
We apply \eqref{noavMQ6} with $\zeta=0$ and use \eqref{largebarxi} which implies $1\lesssim\ln(1+\xi/2)$ to obtain
\begin{equation}\label{noavMQ8}
\dual{\mathcal{M}(\xi/2)}{\mathcal{M}(\xi/2)}_T\lesssim \frac{1}{\xi}\mathcal{U}(T)\mathcal{U}_\Psi(T;0)^2.
\end{equation}
By the Burkholder-Davis-Gundy inequality,
\[
\E\left[\mathcal{M}(T;\xi/2)^*\right]\lesssim\frac{1}{\xi^{1/2}}\E\left[\mathcal{U}(T)^{1/2}\mathcal{U}_\Psi(T;0)\right],
\]
and thus, by the Cauchy-Schwarz inequality and the entropy estimate \eqref{entropyestimatep} with $p=1$,
\begin{equation}\label{p0Martingale0}
\E\left[\mathcal{M}(T;\xi/2)^*\right]\lesssim\frac{1}{\xi^{1/2}}\left\{\E\left[\mathcal{E}(0)\right]\right\}^{1/2}\left\{\E\left[\mathcal{U}_\Psi(T;0)^2\right]\right\}^{1/2}.
\end{equation}
To bound from above the last factor in \eqref{p0Martingale0}, we use the estimate
\begin{equation}\label{PsibyPhi}
|\Psi(a)|^2\lesssim|\Phi(a)|^2,
\end{equation}
which follows from \eqref{PsibyPhi2} below. We exploit then the $L^2$-estimate \eqref{mainestim} to get
\begin{equation}\label{p0Martingale1}
\E\left[\mathcal{M}(T;\xi/2)^*\right]\lesssim\frac{1}{\xi^{1/2}}\left\{\E\left[\mathcal{E}(0)\right]\right\}^{1/2}\left\{\mathtt{E}_2(0)\right\}^{1/2}.
\end{equation}
The estimate of $\E\left[\mathcal{U}_\psi(T;\xi/4)^2\right]$ is done as follows: we observe that $\Psi$ is non-increasing, so that
\begin{equation}\label{PsixibyPhixi}
|\Psi(a_i^\xi)|^2\leq |\Psi(a_i)|^2\mathbf{1}_{a_i>\xi}\leq |\Psi(a_i)|^2\frac{\ln(1+a_i)}{\ln(1+\xi)}\lesssim\frac{|\Phi(a_i)^2|}{\ln(1+\xi)}.
\end{equation}
In the last inequality of \eqref{PsixibyPhixi}, we use the bound
\begin{equation}\label{PsibyPhi2}
|\Psi(a)|^2\ln(1+a)\lesssim|\Phi(a)|^2,
\end{equation}
which is deduced from the obvious estimate 
\[
\Psi(a)\leq\int_0^a\sqrt{\ln(1+s)}ds\leq a\sqrt{\ln(1+a)}.
\]
Then \eqref{PsixibyPhixi} implies 
\begin{equation}\label{p0Markov1}
\E\left[\mathcal{U}_\psi(T;\xi/4)^2\right]\lesssim \frac{\mathtt{E}_2(0)}{\ln(1+\xi)}.
\end{equation}
Set $\mathtt{B}_0=\sup_{1\leq i\leq 4}\|a_{i0}\|_{L^\infty({\T^d})}$. By \eqref{Markov000}, \eqref{p0Martingale1} and \eqref{p0Markov1}, we obtain
\begin{equation}\label{Markov001}
\PP\left[\mathcal{U}_0\geq \delta \right]
\lesssim
\frac{1}{\delta\ln(1+\xi)}\mathtt{E}_2(0)
+\frac{1}{\delta\xi^{1/2}}\mathtt{E}(0)^{1/2}\mathtt{E}_2(0)^{1/2}\leq C(\mathtt{B}_0) \frac{1}{\delta\ln(1+\xi)},
\end{equation}
and, finally, the estimate
\begin{equation}\label{Tailp0}
p_0\geq 1-  \frac{C(\mathtt{B}_0)}{\delta\ln(1+\xi)}.
\end{equation}	
\paragraph{Conclusion.} Let
\begin{equation}
\mathtt{B}_T=\sup_{1\leq i\leq 4}\|a_{i}\|_{L^\infty(Q_T)}.
\end{equation}
We have $\mathtt{B}_T\leq\xi$ if $\mathbf{H}$ is realized, so \eqref{GoodEventUp0U} and \eqref{Tailp0} imply the tail estimate
\begin{equation}\label{tailxi}
\PP(\mathtt{B}_T>\xi)\leq \frac{C(\mathtt{B}_0)}{\delta\ln(1+\xi)}
+C_7\exp\left(-C_5^{-1}\delta^{-(3-2\underline{\rho})}\right)
+C_8\exp\left(-C_3^{-1}\xi^3\right).
\end{equation}
We take
\begin{equation}\label{deltaxi}
\delta=\frac{1}{C_6\vee(\ln(\xi))^{1/2}},
\end{equation}
which satisfies the condition~\eqref{smalldeltassub2}, to obtain (under the condition \eqref{largebarxi})
\begin{equation}\label{tailxi2}
\PP(\mathtt{B}_T>\xi)\leq \frac{\tilde{C}(\mathtt{B}_0)}{\ln(1+\xi)^\frac12}
+C_7\exp\left(-C_9^{-1}\xi^\gamma\right)
+C_8\exp\left(-C_3^{-1}\xi^3\right),
\end{equation}
where $\gamma=\frac{3}{2}-\underline{\rho}>0$. Let $A_0=4(1\vee\mathtt{B}_0)$. The desired result \eqref{eq:Linftyestimate} then follows from \eqref{tailxi2} and the expression
\begin{equation}\label{ETheta}
\E\left[\Theta(\mathtt{B}_T)\mathbf{1}_{\{\mathtt{B}_T\geq A_0\}}\right]
=\int_{A_0}^\infty\Theta'(\xi)\PP(\mathtt{B}_T>\xi)d\xi.
\end{equation}
\end{proof}

\subsection{Existence of global-in-time regular solutions}\label{subsec:globalSol}
 
 The proof of Theorem~\ref{th:globalregular} is now straightforward. We consider the regular solution $a$ constructed in Section~\ref{subsec:globalSol}, based on the sequence of solution $(a^n)$ of the problem with truncated non-linearities. By 
\eqref{eq:Linftyestimate} and the Markov inequality, we have
\begin{equation}
\PP(\tau_n\leq T)\leq\frac{C_\infty}{\Theta(n)},
\end{equation}
which yields $\PP(\tau\leq T)=0$ at the limit $n\to+\infty$. This being true for every $T$, we have $\tau=+\infty$ a.s.

\appendix

\section{Regular solutions to semilinear stochastic parabolic sys\-tems}\label{sec:app1}

We consider the following system of SPDEs: for $1\leq i\leq q$ and $1\leq\alpha\leq d_W$, let $(W^i_\alpha(t))$ be a family of one-dimensional Wiener processes such that, for each $i$, $W^i_1,\dotsc,W^i_{d_W}$ are jointly independent. For $1\leq i\leq q$, let $F_i\colon\R^q\to\R$ and $g_{i,\alpha}\colon\R^q\to\R$ be some given functions and let $\kappa_i$ be some (strictly) positive coefficients. We consider the system

\begin{equation}\label{SSapp}
\begin{split}
 d a_i-\kappa_i\Delta a_i dt&=F_i(a)dt+\sum_{\alpha=1}^{d_W}g_{i,\alpha}(a)d W_\alpha^i(t),\; {\rm in}\  \;Q_T:= {\T^d}\times(0,T),\\
a_i(0)&=a_{i0}\geq 0,\; {\rm in}\   {\T^d},
\end{split}
\end{equation}
for $i=1,\dotsc,q$, where $a=(a_i)_{1,q}$, $a_0=(a_{i0})$ is a given function ${\T^d}\to\R^q$ and $T>0$. Weak solutions are defined as follows.

\begin{definition}[Weak solution to \eqref{SSapp}]\label{def:weaksolSSapp} A process $a\in L^2(\Omega\times[0,T];\mathcal{P};H^1({\T^d}))$ (where $\mathcal{P}$ is the predictable $\sigma$-algebra) is said to be a weak solution to \eqref{SSapp} if
\begin{enumerate}
\item $F_i(a),g_{i,\alpha}(a)\in L^2(\Omega\times Q_T)$,
\item $a_i\in L^2(\Omega;C([0,T];L^2({\T^d})))$,
\item for all $t\in[0,T]$, for all $\varphi\in H^1({\T^d})$, $\PP$-a.s., 
\begin{equation}\label{eq:weaksolapp}
\dual{a_i(t)}{\varphi}=\dual{a_{i0}}{\varphi}\kappa_i-\int_0^t\dual{\nabla a_i(s)}{\varphi}ds+\int_0^t\dual{F_i(a)}{\varphi}ds
+\sum_{\alpha=1}^{d_W}\dual{g_{i,\alpha}(a)}{\varphi}d W_\alpha^i(t).
\end{equation}
\end{enumerate}
\end{definition}

We can then state the following result.

\begin{theorem}[Semilinear stochastic parabolic systems]\label{th:RegSolSyst} Let $p\in[2,\infty)$, $q\in(2,\infty)$ and let $m$ be an integer $\geq 2$ such that $mp>d+2$. Assume 
\begin{equation}
u_0\in W^{m,p}({\T^d})\cap W^{1,mp}({\T^d}),
\end{equation}
and suppose that $F_i\colon\R^q\to\R$ and $g_{i,\alpha}\colon\R^q\to\R$ are functions of class $C^m$, bounded, with all their derivatives up to order $m$ bounded. Then \eqref{SSapp} admits a unique weak solution which belongs to the space
\begin{equation}
L^q(\Omega;C([0,T];W^{m,p}({\T^d})))\cap L^{mq}(\Omega;C([0,T];W^{1,mp}({\T^d}))).
\end{equation}
\end{theorem}

\begin{proof}[Proof of Theorem~\ref{th:RegSolSyst}] The statement of Theorem~\ref{th:RegSolSyst} is essentially the statement of Theorem~2.1 in \cite{Hofmanova13a}. The extension to systems does not raise particular difficulties. The essential point is a generalization of the estimate
\begin{equation}
\|G(h)\|_{W^{m.p}}\leq C\left(1+\|h\|_{W^{1,mp}}^m+\|h\|_{W^{m,p}}\right)
\end{equation}
in \cite[Proposition~3.1]{Hofmanova13a}, given for $h$ real-valued and $G$ defined on $\R$, to the case where $h$ takes values in $\R^q$ and $G$ is defined on $\R^q$. This is obtained simply by generalizing the chain-rule formula
\begin{equation}
D^\beta G(h(x))=\sum_{l=1}^{|\beta|}\sum_{\substack{\alpha_1+\dotsb+\alpha_l=\beta \\ \alpha_i\not=0}}C_{\beta,l,\alpha_1,\dotsb,\alpha_l}G^{(l)}(h(x)) D^{\alpha_1}h(x)\dotsb D^{\alpha_l}h(x)
\end{equation}
to
\begin{multline}
D^\beta G(h(x))=\sum_{l=1}^{|\beta|}\sum_{\substack{\alpha_1+\dotsb+\alpha_l=\beta \\ \alpha_i\not=0}}C_{\beta,l,\alpha_1,\dotsb,\alpha_l}\\
\times\sum_{|\gamma|=l}D^\gamma G(h(x))\sum_{\substack{\delta_1+\dots+\delta_l=\gamma \\ |\delta_i|=1}} D^{\alpha_1}h_{\delta_1}(x)\dotsb D^{\alpha_l}h_{\delta_l}(x),
\end{multline}
where, given $\delta\in\N^q$ of length $|\delta|=1$, $h_{\delta}$ denote the component $h_i$, where $i$ is the only index such that $\delta_i\not=0$.
\end{proof}



\begin{thebibliography}{DDMH15}

\bibitem[BK22]{BendahmaneKarlsen2022}
Mostafa Bendahmane and Kenneth~H. Karlsen.
\newblock Martingale solutions of stochastic nonlocal cross-diffusion systems.
\newblock {\em Netw. Heterog. Media}, 17(5):719--752, 2022.

\bibitem[CDF14]{CanizoDesvillettesFellner14}
J.~A. Ca{\~n}izo, L.~Desvillettes, and K.~Fellner.
\newblock Improved duality estimates and applications to reaction-diffusion
  equations.
\newblock {\em Comm. Partial Differential Equations}, 39(6):1185--1204, 2014.

\bibitem[CGV19]{CaputoGoudonVasseur2019}
M.~Cristina Caputo, Thierry Goudon, and Alexis~F. Vasseur.
\newblock Solutions of the 4-species quadratic reaction-diffusion system are
  bounded and {$C^\infty$}-smooth, in any space dimension.
\newblock {\em Anal. PDE}, 12(7):1773--1804, 2019.

\bibitem[CO20]{ChevallierOst2020}
Julien Chevallier and Guilherme Ost.
\newblock Fluctuations for spatially extended {H}awkes processes.
\newblock {\em Stochastic Process. Appl.}, 130(9):5510--5542, 2020.

\bibitem[CV09]{CaputoVasseur2009}
M.~Cristina Caputo and Alexis Vasseur.
\newblock Global regularity of solutions to systems of reaction-diffusion with
  sub-quadratic growth in any dimension.
\newblock {\em Comm. Partial Differential Equations}, 34(10-12):1228--1250,
  2009.

\bibitem[DDMH15]{DebusscheDeMoorHofmanova15}
A.~Debussche, S.~De~Moor, and M.~Hofmanov{\'a}.
\newblock A regularity result for quasilinear stochastic partial differential
  equations of parabolic type.
\newblock {\em SIAM J. Math. Anal.}, 47(2):1590--1614, 2015.

\bibitem[DFPV07]{DesvillettesFellnerPierreVovelle07}
L.~Desvillettes, K.~Fellner, M.~Pierre, and J.~Vovelle.
\newblock Global existence for quadratic systems of reaction-diffusion.
\newblock {\em Adv. Nonlinear Stud.}, 7(3):491--511, 2007.

\bibitem[DJZ19]{DhariwalJuengelZamponi2019}
Gaurav Dhariwal, Ansgar J\"{u}ngel, and Nicola Zamponi.
\newblock Global martingale solutions for a stochastic population
  cross-diffusion system.
\newblock {\em Stochastic Process. Appl.}, 129(10):3792--3820, 2019.

\bibitem[DMFL86]{DeMasiFerrariLebowitz1986}
A.~De~Masi, P.~A. Ferrari, and J.~L. Lebowitz.
\newblock Reaction-diffusion equations for interacting particle systems.
\newblock {\em J. Statist. Phys.}, 44(3-4):589--644, 1986.

\bibitem[DNN21]{DebusscheNguepedjaNankep2021}
Arnaud Debussche and Mac~Jugal Nguepedja~Nankep.
\newblock A piecewise deterministic limit for a multiscale stochastic spatial
  gene network.
\newblock {\em Appl. Math. Optim.}, 84:S1731--S1767, 2021.

\bibitem[DRV21]{DebusscheRoselloVovelle2021}
Arnaud Debussche, Angelo Rosello, and Julien Vovelle.
\newblock Diffusion-approximation for a kinetic spray-like system with random
  forcing.
\newblock {\em Discrete Contin. Dyn. Syst. Ser. S}, 14(8):2751--2803, 2021.

\bibitem[DT12]{DuTang2012}
Kai Du and Shanjian Tang.
\newblock Strong solution of backward stochastic partial differential equations
  in {$C^2$} domains.
\newblock {\em Probab. Theory Related Fields}, 154(1-2):255--285, 2012.

\bibitem[DYZ23]{DuYeZhang2023}
Yanyan Du, Ming Ye, and Qimin Zhang.
\newblock Global martingale solutions to stochastic population-toxicant model
  with cross-diffusion.
\newblock {\em Appl. Math. Lett.}, 145:Paper No. 108721, 7, 2023.

\bibitem[EK86]{EthierKurtz86}
S.~N. Ethier and T.~G. Kurtz.
\newblock {\em Markov processes}.
\newblock Wiley Series in Probability and Mathematical Statistics: Probability
  and Mathematical Statistics. John Wiley \& Sons Inc., New York, 1986.
\newblock Characterization and convergence.

\bibitem[Fla90]{Flandoli1990}
Franco Flandoli.
\newblock Dirichlet boundary value problem for stochastic parabolic equations:
  compatibility relations and regularity of solutions.
\newblock {\em Stochastics Stochastics Rep.}, 29(3):331--357, 1990.

\bibitem[FMT20]{FellnerMorganTang2020}
Klemens Fellner, Jeff Morgan, and Bao~Quoc Tang.
\newblock Global classical solutions to quadratic systems with mass control in
  arbitrary dimensions.
\newblock {\em Ann. Inst. H. Poincar\'{e} Anal. Non Lin\'{e}aire},
  37(2):281--307, 2020.

\bibitem[Ger19]{Gerencser2019}
M\'{a}t\'{e} Gerencs\'{e}r.
\newblock Boundary regularity of stochastic {PDE}s.
\newblock {\em Ann. Probab.}, 47(2):804--834, 2019.

\bibitem[GV10]{GoudonVasseur10}
T.~Goudon and A.~Vasseur.
\newblock Regularity analysis for systems of reaction-diffusion equations.
\newblock {\em Ann. Sci. \'Ec. Norm. Sup\'er. (4)}, 43(1):117--142, 2010.

\bibitem[Hof13]{Hofmanova13a}
M.~Hofmanov{\'a}.
\newblock Strong solutions of semilinear stochastic partial differential
  equations.
\newblock {\em NoDEA Nonlinear Differential Equations Appl.}, 20(3):757--778,
  2013.

\bibitem[HRT20]{HausenblasRandrianasoloThalhammer20}
Erika Hausenblas, Tsiry~Avisoa Randrianasolo, and Mechtild Thalhammer.
\newblock Theoretical study and numerical simulation of pattern formation in
  the deterministic and stochastic {G}ray-{S}cott equations.
\newblock {\em J. Comput. Appl. Math.}, 364:112335, 27, 2020.

\bibitem[KR79]{KrylovRozovskii1979}
N.~V. Krylov and B.~L. Rozovski\u{\i}.
\newblock Stochastic evolution equations.
\newblock In {\em Current problems in mathematics, {V}ol. 14 ({R}ussian)},
  pages 71--147, 256. Akad. Nauk SSSR, Vsesoyuz. Inst. Nauchn. i Tekhn.
  Informatsii, Moscow, 1979.

\bibitem[Kun15]{Kunze2015}
Markus~C. Kunze.
\newblock Stochastic reaction-diffusion systems with {H}\"{o}lder continuous
  multiplicative noise.
\newblock {\em Stoch. Anal. Appl.}, 33(2):331--355, 2015.

\bibitem[LLN20]{LimLuNolen2020}
Tau~Shean Lim, Yulong Lu, and James~H. Nolen.
\newblock Quantitative propagation of chaos in a bimolecular chemical
  reaction-diffusion model.
\newblock {\em SIAM J Math Anal}, 52(2):2098--2133, 2020.

\bibitem[LV23]{LeocataVovelle2023}
Marta Leocata and Julien Vovelle.
\newblock {Supremum estimates for parabolic stochastic partial differential
  equations}.
\newblock working paper or preprint, November 2023.

\bibitem[NPY21]{NziPardouxYeo2021}
M.~N'zi, E.~Pardoux, and T.~Yeo.
\newblock A {SIR} model on a refining spatial grid {I}: {L}aw of large numbers.
\newblock {\em Appl. Math. Optim.}, 83(2):1153--1189, 2021.

\bibitem[Pie10]{Pierre10}
M.~Pierre.
\newblock Global existence in reaction-diffusion systems with control of mass:
  a survey.
\newblock {\em Milan J. Math.}, 78(2):417--455, 2010.

\bibitem[PS23]{PierreSchmitt2023}
Michel Pierre and Didier Schmitt.
\newblock Examples of finite time blow up in mass dissipative
  reaction-diffusion systems with superquadratic growth.
\newblock {\em Discrete Contin. Dyn. Syst.}, 43(3-4):1686--1701, 2023.

\bibitem[Sou18]{Souplet2018}
Philippe Souplet.
\newblock Global existence for reaction-diffusion systems with dissipation of
  mass and quadratic growth.
\newblock {\em J. Evol. Equ.}, 18(4):1713--1720, 2018.

\end{thebibliography}

\def\cprime{$'$}


\end{document}